\documentclass[a4paper,reqno]{amsart}

\title{City products of right-angled buildings and their universal groups}

\author[J. Bossaert]{Jens Bossaert}

\author[T. De Medts]{Tom De Medts}
\email[Tom De Medts (corresponding author)]{tom.demedts@ugent.be}
\address[Address for both authors]{Ghent University \\ Department of Mathematics: Algebra and Geometry \\ Krijgslaan~281, S25, 9000 Ghent, Belgium}

\date{\today}

\usepackage{amsmath,amssymb,amsthm,mathtools,mathrsfs,graphicx,booktabs,numprint,multirow,url}
\usepackage[labelfont=bf,labelformat=simple]{subcaption}
    
\usepackage{tikz,tikzpagenodes,tikz-3dplot}
	\usetikzlibrary{calc, decorations, decorations.pathmorphing, decorations.markings, shapes.misc, cd}
	\tikzset{
    	myvertex/.style = {circle, draw, fill, thick, outer sep=3\pgflinewidth, inner sep=0pt, minimum size=3.5pt}, 
    	mydoublevertex/.style = {myvertex, double, double distance=\pgflinewidth, minimum size=4pt-2\pgflinewidth, outer sep=4\pgflinewidth},
    	myedge/.style = {draw=black, thick, line cap=round, line join=round},
    	every label/.style = {font=\strut, label distance=0pt, outer sep=0pt, inner sep=0pt},
    	baseline = {([yshift=-.57ex]current bounding box.center)},
    	x=12mm, y=12mm,
    }
    
    \pgfdeclaredecoration{cheating dash}{start}{
    	\state{start}[width=\pgfdecoratedremainingdistance, persistent precomputation={
    		\let\on=\pgfdecorationsegmentlength%
    		\let\off=\pgfdecorationsegmentamplitude%
    		\pgfmathparse{\pgfdecoratedremainingdistance-\on}\let\rest=\pgfmathresult%
    		\pgfmathparse{\on+\off}\let\onoff=\pgfmathresult%
    		\pgfmathparse{max(floor(\rest/\onoff), 1)}\let\nfullonoff=\pgfmathresult%
    		\pgfmathparse{max((\rest-\onoff*\nfullonoff)/\nfullonoff+\off, \off)}\global\let\offexpand=\pgfmathresult%
    	}]{\pgfsetpath\pgfdecoratedpath}
    }
    
    \definecolor{ugentblue}{RGB}{30,100,200}
    \definecolor{ugentyellow}{RGB}{255,210,0}
    \definecolor{ugentsecondary}{RGB}{190,81,144}
    
    \definecolor{ugentred}{RGB}{220,78,40}
    \definecolor{ugentgreen}{RGB}{113,168,96}
    
    \definecolor{mylightgray}{gray}{.94}
    \definecolor{mydarkgray}{gray}{.86}
    
    \tikzset{cheating dash/.code args={on #1 off #2}{
    	\tikzset{decoration={cheating dash, segment length=#1, amplitude=#2}, decorate}%
    	\csname tikz@addoption\endcsname{\pgfsetdash{{#1}{\offexpand}}{0pt}}}
    }
    
    \tikzset{
    	dotted/.style = {very thick, cheating dash = on 2000sp off 3.6pt},
    	dashed/.style = {cheating dash = on 3pt off 3pt},
    	-dots/.style = {postaction = {decoration = {markings, mark=at position 1 with {\node[anchor=180, transform shape, outer sep=2pt, inner sep=0pt] {\kern45000sp\textellipsis};}}, decorate}}
    }
    
    \tikzcdset{every arrow/.append style={line width=.45pt},
    	cells = {nodes = {rounded rectangle, font=\strut, outer sep=2pt, inner sep=2pt}},
    	dims/.style 2 args = {column sep = {#1,between origins}, row sep = {#2,between origins}},
    	dims/.default = {5em}{5em},
    	diagrams = {dims},
    }
    
    \tikzcdset{
    	relay arrow/.default = 10pt,
    	relay arrow/.style = {rounded corners, to path = {
    		\pgfextra{
    			\def\sourcecoordinate{\pgfpointanchor{\tikztostart}{center}}
    			\def\targetcoordinate{\pgfpointanchor{\tikztotarget}{center}}
    			\pgfmathanglebetweenpoints{\sourcecoordinate}{\targetcoordinate}
    			\edef\tempangle{\pgfmathresult}
    			\pgftransformrotate{\tempangle}
    			\pgfmathifthenelse{#1>0}{\tempangle+90}{\tempangle-90}
    			\pgfcoordinate{tempcoord}{\pgfpointanchor{\tikztostart}{\pgfmathresult}}
    		}
    		(tempcoord) -- ([yshift=#1]tempcoord)
    			-- ([yshift=#1]tempcoord-|\tikztotarget.center) \tikztonodes
    			-- (\tikztotarget)
    	}}
    }

\makeatletter
\newcommand*\rel@kern[1]{\kern#1\dimexpr\macc@kerna}
\newcommand*\widebar[1]{%
  \begingroup
  \def\mathaccent##1##2{%
    \rel@kern{0.8}%
    \overline{\rel@kern{-0.8}\macc@nucleus\rel@kern{0.2}}%
    \rel@kern{-0.2}%
  }%
  \macc@depth\@ne
  \let\math@bgroup\@empty \let\math@egroup\macc@set@skewchar
  \mathsurround\z@ \frozen@everymath{\mathgroup\macc@group\relax}%
  \macc@set@skewchar\relax
  \let\mathaccentV\macc@nested@a
  \macc@nested@a\relax111{#1}%
  \endgroup
}
\makeatother

\usepackage{enumitem}
	\setenumerate[0]{label={\rm (\roman*)},leftmargin=*,itemsep=.2ex}
\usepackage[symbol]{footmisc}
\usepackage{stackengine}
\usepackage{hyperref}
	\hypersetup{bookmarksopen=true}
\usepackage[capitalize]{cleveref}
    \newcommand{\itemref}[2]{\cref{#1}\ref{#1:#2}}

\usepackage{pgfplots, pgfplotstable}
\pgfplotsset{compat=1.15}

\renewcommand\footnotemark{}

\newtheorem{theorem}{Theorem}[section]
\newtheorem{lemma}[theorem]{Lemma}
\newtheorem{proposition}[theorem]{Proposition}
\newtheorem{corollary}[theorem]{Corollary}

\newtheoremstyle{step}{\topsep}{\topsep}{\itshape}{0pt}{\itshape}{\@. }{0pt}{\thmname{#1}\thmnumber{ #2}\thmnote{ (#3)}}
\theoremstyle{step}

\theoremstyle{definition}
\newtheorem{definition}[theorem]{Definition}
\newtheorem{remark}[theorem]{Remark}
\newtheorem{example}[theorem]{Example}
\newtheorem{examples}[theorem]{Examples}

\DeclareMathOperator{\Sym}{Sym}
\DeclareMathOperator{\Aut}{Aut}

\DeclareMathOperator{\U}{\mkern2mu\mathcal{U}}

\makeatletter

\newcommand{\@Fhack}{\@ifnextchar,{\mkern-2mu}{}}
\newcommand{\F}{\ensuremath\boldsymbol{F}\@Fhack}
\newcommand{\bL}{\ensuremath\boldsymbol{L}\@Fhack}
\newcommand{\Facute}{\ensuremath\boldsymbol{\smash{\acute F}}\@Fhack}
\newcommand{\Fwidehat}{\ensuremath\boldsymbol{\widehat F}\@Fhack}

\renewcommand{\P}{\ensuremath\mathcal{P}}
\newcommand{\R}{\ensuremath\mathcal{R}}

\newcommand{\cityprod}{\ensuremath\operatorname{\rotatebox[x=.5ex,y=.3ex]{90}{$\bowtie$}}}

\DeclareMathOperator{\Res}{Res}

\DeclareMathOperator{\dist}{dist}

\newcommand\restrict[3][]{\mathchoice{{#2\bigr\rvert_{#3}^{#1}}}{{#2\rvert_{#3}^{#1}}}{{#2\rvert_{#3}^{#1}}}{{#2\rvert_{#3}^{#1}}}}

\newcommand\acts{\mkern 2mu\relax.\mkern 2mu\relax}

\@namedef{subjclassname@2020}{\textup{2020} Mathematics Subject Classification}

\makeatother


\begin{document}

\begin{abstract}
    We introduce the notion of \emph{city products} of right-angled buildings that produces a new right-angled building out of smaller ones.
    More precisely, if $M$ is a right-angled Coxeter diagram of rank $n$ and $\Delta_1,\dots,\Delta_n$ are right-angled buildings, then we construct a new right-angled building $\Delta := \cityprod_M(\Delta_1,\dots,\Delta_n)$. We can recover the buildings $\Delta_1,\dots,\Delta_n$ as residues of $\Delta$, but we can also construct a \emph{skeletal building} of type $M$ from $\Delta$ that captures the large-scale geometry of $\Delta$.
    
    We then proceed to study \emph{universal groups} for city products of right-angled buildings, and we show that the universal group of $\Delta$ can be expressed in terms of the universal groups for the buildings $\Delta_1,\dots,\Delta_n$ and the structure of $M$. As an application, we show the existence of many examples of pairs of different buildings of the same type that admit (topologically) isomorphic universal groups, thereby vastly generalizing a recent example by Lara Be{\ss}mann.
\end{abstract}

\dedicatory{In memory of Jacques Tits, architect of buildings}

\maketitle


\section{Introduction}

A building is called \emph{right-angled} if its Coxeter group is right-angled, which means that the only values occurring in its Coxeter matrix are $1$, $2$ and $\infty$.
The prototypical example is the case where the Coxeter matrix has rank $2$ with a label $\infty$, in which case the building is a tree. In general, the behavior of right-angled buildings is somewhat comparable to that of trees, but in a combinatorially much more complicated (and therefore much more interesting) way.

The first systematic study of right-angled buildings is by Fr\'ed\'eric Haglund and Fr\'ed\'eric Paulin \cite{haglundpaulin}, who showed the existence and uniqueness of right-angled buildings for any set of parameters (see \cref{thm:rabsexist} below).
Later, right-angled buildings have been used to construct interesting examples of \emph{lattices}, as in the work of Angela Kubena, Anne Thomas and Kevin Wortman \cite{thomas1,thomas3,thomas2}.

Our motivation for studying right-angled buildings, initiated by Pierre-Emmanuel Caprace in \cite{caprace2014}, is the connection with totally disconnected locally compact groups. More precisely, the automorphism group of a right-angled building is always totally disconnected with respect to the permutation topology, and if the building is locally finite, then the automorphism group is also locally compact. This is not true in general, but these automorphism groups contain lots of interesting subgroups, namely so-called \emph{universal groups}, that can still be locally compact even if the building is not locally finite.

These universal groups were first introduced and studied for trees by Marc Burger and Shahar Mozes in their seminal paper \cite{burgermozes}. 
This concept has been generalized to right-angled buildings by the second author in joint work with Ana C. Silva and Koen Struyve in \cite{silva1,silva2} in the locally finite case, and has been further generalized and studied without this assumption in our paper \cite{bossaert}, focussing on topological properties.

\medskip

For some right-angled buildings, the large-scale geometry looks like a tree; see, for instance, \cref{fig:coxeterrab} below. This raises the question whether it is possible, in these cases, to somehow reverse the process, i.e., whether we can start from a tree and obtain a more complicated right-angled building by ``inserting'' more complicated blocks at each vertex of the tree.

This idea gave rise to the construction that we introduce and study in this paper. We call it the \emph{city product} of buildings, as it is a way to construct larger objects out of a given number of buildings, guided by the rough structure of yet another right-angled diagram.
More precisely, if $M$ is a right-angled Coxeter diagram of rank $n$ and $\Delta_1,\dots,\Delta_n$ are right-angled buildings, then we construct a new right-angled building $\Delta := \cityprod_M(\Delta_1,\dots,\Delta_n)$. We can recover the buildings $\Delta_1,\dots,\Delta_n$ as residues of $\Delta$, but we can also construct a \emph{skeletal building} $\Phi$ of type $M$ from $\Delta$ that captures the large-scale geometry of $\Delta$.
Constructing this building $\Phi$ is not difficult, but it turns out to be far from trivial to show that it is indeed a building. This is the content of \cref{prop:cityproductskeletalstuff}, which relies on the new notions of \emph{weak homotopies} and \emph{parkour maps} that we have introduced for this purpose.

\medskip

It turns out that the universal groups for these city products can be described as the universal group of this skeletal building $\Phi$ with respect to universal groups for each of the smaller buildings $\Delta_1,\dots,\Delta_n$; this is the content of \cref{thm:universalcityproduct}.

\medskip

In a recent preprint \cite{bessmann}, Lara Be{\ss}mann has shown the existence of pairs of different right-angled buildings, both of type
\raisebox{.8ex}{%
	\begin{tikzpicture}[scale=.6]
	   \useasboundingbox (-.2,.1) rectangle (2.2,.4);
		\path (0,0) node[myvertex] (A) {}
			++(1,0) node[myvertex] (B) {}
			++(1,0) node[myvertex] (C) {};
		\draw[myedge] (A) -- node[above] {\footnotesize $\infty$} (B) -- node[above] {\footnotesize $\infty$} (C);
	\end{tikzpicture}},
admitting universal groups that are topologically isomorphic.
Her method relies on the notion of tree-wall trees from \cite{silva1} and only works for star-shaped diagrams (see \itemref{ex:isom}{1}).
We show that this can be interpreted in terms of city products, which allows us to produce many more examples of such pairs. This is the content of \cref{thm:application}.

\subsection*{Acknowledgment.}

The first author has been supported by the UGent BOF PhD mandate BOF17/DOC/274.
We thank an anonymous referee for carefully reading the paper and suggesting several improvements in the exposition.



\section{Preliminaries}

\subsection{Coxeter systems}

\begin{definition}\label{def:cox}
    \begin{enumerate}
        \item 
        	Let $I$ be any index set and $M$ a function
        	\[M\colon I\times I\to \mathbb N\cup\{\infty\}\colon (i,j)\mapsto m_{ij}\]
        	satisfying $m_{ii}=1$, $m_{ij}\geq 2$, and $m_{ij}=m_{ji}$ for all $i\neq j\in I$. Then the \emph{Coxeter group of type $M$} is the group defined by the presentation
        	\[W = \bigl\langle s_i \mathrel{\bigm|} \text{$(s_i s_j)^{m_{ij}} = 1$ for all $i,j\in I$}\bigr\rangle.\]
        	When $m_{ij}=\infty$, this means that no relation on $s_is_j$ should be imposed. Note that the assumption that $m_{ii}=1$ for all $i\in I$ immediately implies that the generators $s_i$ are involutions. Additionally, note that when $m_{ij}=2$, the generators $s_i$ and $s_j$ commute.
        	
        	Together with the generating set $S=\{s_i\mid i\in I\}$, the pair $(W,S)$ is called the \emph{Coxeter system of type $M$}. The \emph{rank} of $(W,S)$ is the cardinality of $I$.
        	
        	We can represent $M$ by means of its \emph{Coxeter matrix $(m_{ij})$}, or more commonly its \emph{Coxeter diagram}: the nodes of the diagram are the elements of $I$ (sometimes with explicit labels), and two nodes are connected by a decorated edge according to the following rules:
        	\[\newcommand\coxeterdiagram[2]
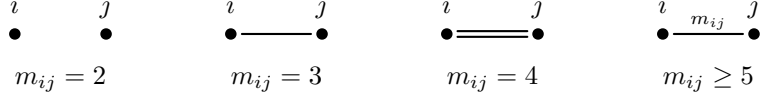
{\begin{tikzpicture}[x=6mm]
        			\node[myvertex,label=above:$i$] (A) at (-1,0) {};
        			\node[myvertex,label=above:$j$] (B) at (1,0) {};
        			\path[myedge] #2;
        			\node[yshift=-6mm] {#1};
        			\end{tikzpicture}}
        		\coxeterdiagram{$m_{ij}=2$}{}
        			\qquad\qquad\coxeterdiagram{$m_{ij}=3$}{(A) -- (B)}
        			\qquad\qquad\coxeterdiagram{$m_{ij}=4$}{(A.15) -- (B.165) (A.345) -- (B.195)}
        			\qquad\qquad\coxeterdiagram{$m_{ij}\geq 5$}{(A) -- node[above] {\smash{\scriptsize $m_{ij}$}} (B)}\]
        \item
        	We call a Coxeter system $(W,S)$ \emph{irreducible} if the underlying graph of its Coxeter diagram is connected, and \emph{reducible} otherwise.
        \item\label{def:cox:RA}
        	We call a Coxeter system $(W,S)$ \emph{right-angled} if $m_{ij}\in\{2,\infty\}$ for all $i\neq j$.
    \end{enumerate}
\end{definition}

In general, non-isomorphic Coxeter systems may have isomorphic Coxeter groups, but this cannot occur for right-angled Coxeter systems:
\begin{theorem}
	\label{thm:racgrigid}
	If a right-angled Coxeter group $W$ admits two Coxeter systems $(W,S)$ and $(W,S')$, then these Coxeter systems are isomorphic (i.e.~there is a diagram-preserving bijection $S\to S'$).
\end{theorem}
\begin{proof}
	We refer to \cite{radcliffe} or \cite{hosaka}.
\end{proof}

\begin{definition}
	\label{def:evaluationmorphism}
	Let $(W,S)$ be a Coxeter system over some index set $I$.
	\begin{enumerate}
	    \item 
    		We will write $I^*$ for the free monoid over $I$. The elements of $I^*$ will be called \emph{words}.
	    \item 
        	There is a natural surjective \emph{evaluation morphism} of monoids 
	\end{enumerate}
	\[\epsilon\colon I^* \to W\colon i\mapsto s_i.\]
\end{definition}
\begin{definition}
	For every $i\neq j$ such that $m_{ij}$ is finite, define in $I^*$ the word
	\[p(i,j) = \begin{cases}
	 	(ij)^k & \text{if $m_{ij}=2k$ is even,}\\
	 	j(ij)^k & \text{if $m_{ij}=2k+1$ is odd.}
	\end{cases}\]
	In other words, $p(i,j)$ is the word with $m_{ij}$ alternating letters $i$ and $j$, ending in $j$. When $m_{ij}=\infty$, $p(i,j)$ remains undefined.
\end{definition}

\begin{definition}
	\label{def:homotopy}
	Let $i,j\in I$ and $w_1,w_2\in I^*$.
	\begin{enumerate}
		\item An \emph{elementary homotopy} (or also a \emph{braid relation}) is a transformation of a word $w_1\,p(i,j)\,w_2$ into the word $w_1\,p(j,i)\,w_2$.
		\item Two words $w$ and $w'$ are \emph{homotopic} if $w$ can be transformed into $w'$ by a sequence of elementary homotopies; we denote this by $w\simeq w'$. Clearly, homotopy is an equivalence relation and preserves the length of the words.
		\item An \emph{elementary contraction} is a transformation of a word $w_1\,ii\,w_2$ into the word $w_1\,w_2$.
		\item An \emph{elementary expansion} is a transformation of a word $w_1\,w_2$ into a word $w_1\,ii\,w_2$.
    	\item A word is called \emph{reduced} if it is not homotopic to a word of the form $w_1\,ii\,w_2$ (for some $i\in I$).
    	\item Two words $w$ and $w'$ are called \emph{equivalent} if $w$ can be transformed into $w'$ by a sequence of elementary homotopies, contractions, and expansions. Clearly, every equivalence class contains some reduced word.
	\end{enumerate}
\end{definition}

\begin{theorem}
	\label{thm:homotopicstuff}
	\begin{enumerate}
		\item\label{thm:homotopicstuff:1} Two words $w$ and $w'$ are equivalent if and only if $\epsilon(w)=\epsilon(w')$.
		\item\label{thm:homotopicstuff:2} Two reduced words $w$ and $w'$ are equivalent if and only if they are homotopic.
		\item\label{thm:homotopicstuff:3} Let $w$ be a reduced word and let $i\in I$. If $iw$ (or $wi$) is not reduced, then $w$ is homotopic to a word that begins (or ends, respectively) with $i$.
	\end{enumerate}
\end{theorem}
\begin{proof}
	By the defining relations $(s_is_j)^{m_{ij}} = 1$ in the presentation, $p(i,j)$ and $p(j,i)$ have the same image under $\epsilon$, and $\epsilon(ii)$ is the identity. Statement \ref{thm:homotopicstuff:1} follows immediately. For \ref{thm:homotopicstuff:2}, we refer to \cite[Theorem 2.11]{ronan}. Statement \ref{thm:homotopicstuff:3} is \cite[Corollary 2.13]{ronan}.
\end{proof}

\subsection{Chamber systems}

Our approach is based on \cite{ronan}.

\begin{definition}
	Let $I$ be any index set. A \emph{chamber system over $I$} is a set $\Delta$ together with, for every $i\in I$, an equivalence relation called \emph{$i$-adjacency}. The elements of $\Delta$ are called \emph{chambers}. If two chambers $c$ and $d$ are $i$-adjacent, we write $c\sim_i d$, or simply $c\sim d$ if we do not want to stress the adjacency type. The cardinality $|I|$ is called the \emph{rank} of $\Delta$. In this paper, the rank will always be finite.
	
	We will usually say that ``$\Delta$ is a chamber system'' when the equivalence relations on~$\Delta$ are clear from the context.
\end{definition}

\begin{definition}
    Let $\Delta$ be a chamber system over $I$.
	A \emph{gallery $\gamma$} in $\Delta$ is a finite sequence of pairwise adjacent chambers
	\[c_0 \sim_{i_1} c_1 \sim_{i_2} {\cdots} \sim_{i_n} c_n\]
	for certain $i_1,\dots,i_n\in I$. We call the word $i_1\mathbin{\cdots} i_n \in I^*$ the \emph{type} of~$\gamma$, and the integer $n$ the \emph{length} of $\gamma$. If there is no strictly shorter gallery from $c_0$ to $c_n$, then we call $\gamma$ a \emph{minimal} gallery.
	
	Chamber systems come equipped with a natural metric
	\[\dist\colon \Delta\times\Delta\to\mathbb N\cup\{\infty\}\]
	defined by declaring $\dist(c,d)$ to be the minimal length of all galleries joining $c$~and~$d$ (or $\infty$ if there is no such gallery). It is clear that this distance function is positive-definite, symmetric, and satisfies the triangle inequality.
\end{definition}

\begin{definition}
	Let $J\subseteq I$. A subset $C\subseteq\Delta$ is called \emph{$J$-connected} if any two chambers in $C$ can be joined by a gallery of type in $J^*$. A \emph{residue of type $J$}, or simply a \emph{$J$-residue}, is a $J$-connected component of $\Delta$. A \emph{panel of type $j$}, or simply a \emph{$j$-panel}, is a residue of type $\{j\}$. The set of all $J$-residues of the chamber system $\Delta$ will be denoted by $\Res_J(\Delta)$.
	
	Note that each $J$-residue is, in its own right, a connected chamber system over the index set $J$.
\end{definition}

\begin{definition}
	A chamber system is called \emph{thin} if every panel contains exactly two chambers, and \emph{thick} if every panel contains at least three chambers. (Panels containing only a single chamber are degenerate cases that should not occur in any reasonable application.)
\end{definition}

Note that a chamber system might be neither thin nor thick.

\begin{definition}
	A map $\varphi\colon\Delta_1\to\Delta_2$ between two chamber systems is a \emph{morphism} if $\varphi(c) \sim \varphi(d)$ in $\Delta_2$ whenever $c \sim d$ in $\Delta_1$. As usual, an \emph{isomorphism} is a bijective morphism, and an \emph{automorphism} is an isomorphism to the same chamber system. Assuming that $\Delta_1$ and $\Delta_2$ have the same index set, a morphism is \emph{type-preserving} if $\varphi(c) \sim_i \varphi(d)$ whenever $c \sim_i d$. In this paper, we shall always assume morphisms to be type-preserving.
	 
	 The set of all automorphisms of a chamber system $\Delta$ forms a group, denoted by $\Aut(\Delta)$.
\end{definition}

\begin{definition}
	Let $(W,S)$ be a Coxeter system of type $M$ over $I$. Define a chamber system over $I$ with the elements of $W$ as chambers, and declare two group elements $v$ and $w$ to be $i$-adjacent if and only if $vs_i = w$. 
	The resulting chamber system is called the \emph{Coxeter complex of type $M$}.
	Coxeter complexes are always connected and thin: every chamber is $i$-adjacent to exactly one other chamber for every $i\in I$.

    Observe that the Coxeter complex associated to a Coxeter system $(W,S)$ is nothing more than the (undirected) Cayley graph of $W$ with respect to the generating set $S$.
\end{definition}
\begin{theorem}\label{thm:coxcpx}
    A gallery in a Coxeter complex is minimal if and only if its type is reduced.
\end{theorem}
\begin{proof}
    See \cite[Theorem 2.11]{ronan}.
\end{proof}

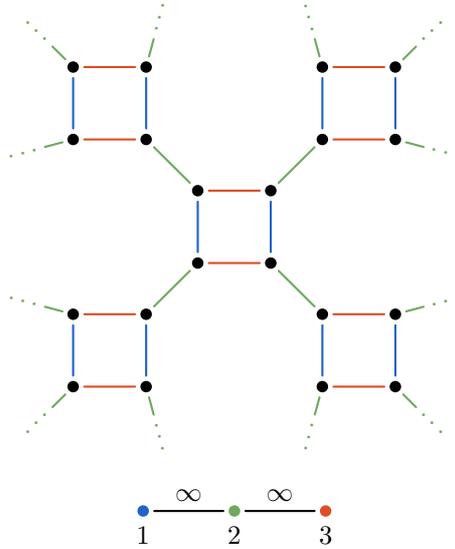
\begin{figure}[!ht]
	\centering
	\begin{tikzpicture}[scale=.8]
		\pgfmathsetmacro{\s}{1+sqrt(2)/2}
		\foreach\P/\M in {(0,0)/C, (-\s,-\s)/LL, (-\s,\s)/UL, (\s,\s)/UR, (\s,-\s)/LR}
			\foreach\Q/\N in {(-.5,-.5)/LL, (-.5,.5)/UL, (.5,.5)/UR, (.5,-.5)/LR}
				\path \P +\Q node[myvertex] (\M-\N) {};
		\begin{scope}[myedge,ugentgreen]
			\draw[-dots] (LL-LR) -- ++(285:.4);
				\draw[-dots] (LL-LL) -- ++(225:.4);
				\draw[-dots] (LL-UL) -- ++(165:.4);
			\draw[-dots] (UL-LL) -- ++(195:.4);
				\draw[-dots] (UL-UL) -- ++(135:.4);
				\draw[-dots] (UL-UR) -- ++(75:.4);
			\draw[-dots] (UR-UL) -- ++(105:.4);
				\draw[-dots] (UR-UR) -- ++(45:.4);
				\draw[-dots] (UR-LR) -- ++(345:.4);
			\draw[-dots] (LR-UR) -- ++(15:.4);
				\draw[-dots] (LR-LR) -- ++(315:.4);
				\draw[-dots] (LR-LL) -- ++(255:.4);
			\draw (C-LL) -- (LL-UR) (C-UL) -- (UL-LR) (C-UR) -- (UR-LL) (C-LR) -- (LR-UL);				
		\end{scope}
		\draw[myedge,ugentred]
			\foreach\M in {C, LL, UL, UR, LR} {(\M-LL) -- (\M-LR) (\M-UL) -- (\M-UR)};
		\draw[myedge,ugentblue]
			\foreach\M in {C, LL, UL, UR, LR} {(\M-LL) -- (\M-UL) (\M-LR) -- (\M-UR)};
	\end{tikzpicture}
	\par\bigskip
	\begin{tikzpicture}
		\path (0,0) node[myvertex,ugentblue,label=below:$1$] (A) {}
			++(1,0) node[myvertex,ugentgreen,label=below:$2$] (B) {}
			++(1,0) node[myvertex,ugentred,label=below:$3$] (C) {};
		\draw[myedge] (A) -- node[above] {$\infty$} (B) -- node[above] {$\infty$} (C);
	\end{tikzpicture}
\caption{A right-angled Coxeter complex}
\label{fig:coxeterrab}
\end{figure}

\subsection{Right-angled buildings}

\begin{definition}
	Let $(W,S)$ be a Coxeter system of type $M$ over some index set~$I$. A \emph{building $(\Delta,\delta)$ of type $M$} is a chamber system $\Delta$ over $I$ such that every panel contains at least two chambers, equipped with a map $\delta\colon\Delta\times\Delta\to W$ satisfying the following property for every reduced word $w\in I^*$:
	\[\text{$\delta(c,d) = \epsilon(w)$ if and only if $c$ and $d$ can be joined by a gallery of type $w$.}\]
	Such a gallery is automatically minimal by \cref{thm:coxcpx}. In particular, the distance between two chambers $c$ and $d$ is exactly the length of $\delta(c,d)$ in the word metric of~$W$ (w.r.t.~generating set $S$).
	
	The group $W$ is called the \emph{Weyl group} of the building, and the map $\delta$ is called the \emph{$W$\!-distance} or \emph{Weyl distance function}. 
	
	We shall usually identify the building with its chamber set and abbreviate $(\Delta,\delta)$ to $\Delta$.
\end{definition}

\begin{definition}
    A building $\Delta$ is called \emph{right-angled} if its underlying Coxeter system $(W,S)$ is right-angled (as defined in \itemref{def:cox}{RA}).
\end{definition}

\begin{definition}
	A building $\Delta$ over $I$ is called \emph{semiregular with parameters $(q_i)_{i\in I}$} if for each $i\in I$, all panels of type $i$ have the same (possibly infinite) cardinality $q_i\geq 2$. Note that the thin buildings are precisely the semiregular buildings with parameters $q_i=2$ for all $i$.
\end{definition}

The following result is attributed to Haglund and Paulin, but they point out that this fact was already known to Mark Globus (but unpublished), Michael Davis and Gabor Moussong, and Tadeusz Januszkiewicz and Jacek \'Swi\c{a}tkowski. 
\begin{theorem}
	\label{thm:rabsexist}
	For any choice of (possibly infinite) cardinal numbers $(q_i)_{i\in I}$ with $q_i \geq 2$, there exists a semiregular right-angled building $\Delta$ with these parameters. Moreover, $\Delta$ is unique up to isomorphism, the automorphism group $\Aut(\Delta)$ acts transitively on the chambers, and every automorphism of a residue of $\Delta$ extends to an automorphism of $\Delta$.
\end{theorem}
\begin{proof}
	See \cite[Proposition 1.2]{haglundpaulin}.
\end{proof}

\subsection{Colorings and implosions of right-angled buildings}

In order to keep track of the local behavior of a building automorphism, it is useful to introduce colorings of the building.
Throughout this section, $\Delta$ is a semiregular right-angled building with parameters $(q_i)_{i\in I}$.
The following notion of legal colorings was introduced in \cite[Definition 2.42]{silva1}.
\begin{definition}
	\label{def:coloring}
	Consider a set $\Omega_i$ of cardinality $q_i$ for each $i\in I$, the elements of which we call \emph{$i$-colors} or \emph{$i$-labels}. A \emph{legal coloring} of $\Delta$ (with color sets $\Omega_i$) is a map
	\[\lambda\colon \Delta\to\prod_{i\in I} \Omega_i\colon c\mapsto(\lambda_i(c))_{i\in I}\]
	satisfying the following properties for every $i\in I$ and for every $i$-panel $\P$:
	\begin{enumerate}
		\item the restriction $\restrict{\lambda_i}{\P}\colon \P\to\Omega_i$ is a bijection;
		\item for every $j\neq i$, the restriction $\restrict{\lambda_j}{\P}\colon \P\to\Omega_j$ is a constant map.
	\end{enumerate}
\end{definition}

Such a legal coloring is essentially unique:
\begin{proposition}
	\label{prop:colortransformation}
	Let $\lambda$ and $\lambda'$ be two legal colorings of a right-angled building $\Delta$ using identical color sets. Let $c$ and $c'$ be two chambers such that $\lambda(c) = \lambda'(c')$. Then there exists an automorphism $g\in\Aut(\Delta)$ such that $g\acts c=c'$ and $\lambda'\circ g = \lambda$.
\end{proposition}
\begin{proof}
    See \cite[Proposition 2.44]{silva1}.
\end{proof}

We now recall the notion of an implosion of a right-angled building, introduced in \cite[Definition 5.2]{bossaert}.
\begin{definition}
	\label{def:implosion1}
	Let $\Delta$ be a semiregular right-angled building over $I$ and let $\lambda$ be a legal coloring of $\Delta$ (using color sets $\Omega_i$).
	For each $i \in I$, consider an equivalence relation $\equiv_i$ on $\Omega_i$, let $\Omega'_i := \Omega_i/{\equiv_i}$ and set $q'_i := \lvert \Omega'_i \rvert$. For each $\lambda_i \in \Omega_i$, we write $[\lambda_i]$ for the corresponding element of $\Omega'_i$.
	Let
	\[ I' = \{i\in I \mid \text{${\equiv_i}$ is not the universal relation}\} = \{ i \in I \mid q'_i \neq 1 \}. \]
	Define a new semiregular right-angled building $\Delta'$ over $I'$ with diagram induced by the diagram of $\Delta$, with parameters $q'_i$ (for every $i\in I'$), and with a legal coloring $\lambda'$ using the quotient $\Omega'_i$ as the set of $i$-colors.
\end{definition}

Recall that a map $f\colon X\to Y$ between metric spaces is called \emph{nonexpansive} if it does not increase distances, i.e., if $\dist_Y(f(x_1),f(x_2)) \leq \dist_X(x_1,x_2)$ for every pair $(x_1,x_2)$ of points in $X$.
\begin{proposition}
	\label{prop:implosion}
	Let $\Delta$, $\lambda$, $\equiv_i$ and $\Delta'$ be as in \cref{def:implosion1}. Let $c_0\in\Delta$ be any chamber and let $c'_0\in\Delta'$ be such that $\lambda'_i(c'_0) = [\lambda_i(c_0)]$ for every $i\in I'$. Then there exists a unique nonexpansive epimorphism $\tau$ of chamber systems from $\Delta$ onto $\Delta'$ mapping $c_0$ to $c'_0$ such that $\lambda'_i(\tau(c)) = [\lambda_i(c)]$ for all $c\in\Delta$.
\end{proposition}
\begin{proof}
    See \cite[Proposition 5.1 and Remark 5.4]{bossaert}.
\end{proof}
\begin{definition}
    We call the pair $(\Delta', \tau)$ from \cref{prop:implosion} the \emph{implosion} of $\Delta$ with \emph{centre} $c_0$ (with respect to the relations $\equiv_i$).
\end{definition}

\begin{corollary}
	\label{cor:residuecoloring}
	Let $\Delta$ be a semiregular right-angled building of type $M$ over $I$, let $J\subseteq I$, and let $\Gamma$ be the semiregular building of type $M_J$ over $J$ with the same parameters as $\Delta$. Then there is a map $\varphi_J\colon\Delta\to\Gamma$ with the following properties:
	\begin{enumerate}
		\item for every residue $\R$ of type $J$, the restriction $\restrict{\varphi_J}{\R}$ is an isomorphism;
		\item for every residue $\R$ of type $I\setminus J$, the restriction $\restrict{\varphi_J}{\R}$ is a constant map.
		\end{enumerate}
\end{corollary}
\begin{proof}
	This follows from \cref{prop:implosion} by taking as equivalence relations $\equiv_i$ either the equality relation if $i\in J$ or the universal relation if $i\notin J$.
\end{proof}

\subsection{Universal groups}

Universal groups for right-angled buildings were first introduced in \cite{silva1} and further studied in the locally finite case in \cite{silva2}.
Their topological properties in the general case have been further investigated in \cite{bossaert}.

\begin{definition}
	\label{def:universal}
    Let $\Delta$ be a semiregular right-angled building over $I$, with parameters $(q_i)_{i \in I}$.
    For each $i$, let $\Omega_i$ be a color set of size $q_i$ and let $\lambda$ be a corresponding legal coloring of $\Delta$.
    \begin{enumerate}
        \item 
            Consider an automorphism $g\in\Aut(\Delta)$ and an arbitrary $i$-panel $\P$. Then we define the \emph{local action of $g$ at $\P$} as the map
        	\[\sigma_\lambda(g,\P) = \restrict{\lambda_i}{g\P} \circ \restrict{g}{\P} \circ \restrict[-1]{\lambda_i}{\P}, \]
        	which is a permutation of $\Omega_i$ by definition of $\lambda$.
        	In other words, the local action $\sigma_\lambda(g,\P)$ is the map that makes the following diagram commute.
        	\[\begin{tikzcd}[dims={7em}{4em}]
        		\P \ar[r,"g"] \ar[d,"\lambda_i"]
        			& g\P \ar[d,"\lambda_i"]\\
        		\Omega_i \ar[r,"{\sigma_\lambda(g,\,\P)}"] & \Omega_i
        	\end{tikzcd}\]
        \item 
        	Let $\F$ be a collection of permutation groups $F_i\leq\Sym(\Omega_i)$, indexed by $i\in I$. The \emph{universal group} of $\F$ over $\Delta$ (with respect to $\lambda$) is the group
        	\[
        	   \U_\Delta^\lambda(\F) = \bigl\{g\in\Aut(\Delta) \bigm| \sigma_\lambda(g,\P)\in F_i  \text{ for each $i\in I$ and each $\P\in\Res_i(\Delta)$}\bigr\}.
        	\]
        	In words, $\U_\Delta^\lambda(\F)$ is the group of automorphisms that locally act like permutations in $F_i$. We hence call the groups $F_i$ the \emph{local groups} and we refer to the collection $\F$ as the \emph{local data} for the universal group. 
    \end{enumerate}
\end{definition}
\begin{remark}\label{rem:Unotation}
\begin{enumerate}
    \item 
    	When the coloring $\lambda$ is clear from the context, we will usually omit the explicit reference to $\lambda$ and simply use the notation $\sigma(g,\P)$ and $\U_\Delta(\F)$ instead.
    	In fact, the choice of $\lambda$ is irrelevant, since different colorings give rise to conjugate subgroups of $\Aut(\Delta)$; see \cite[Proposition 3.7(1)]{silva1}\footnote{The statement of \cite[Proposition 3.7(1)]{silva1} is for \emph{locally finite} right-angled buildings only, but the proof continues to hold for arbitrary right-angled buildings, as pointed out already in \cite[\S 2.3]{bossaert}.}.
    \item\label{rem:Unotation:2}
        When each of the groups $F_i$ in the local data $\F$ is given as a permutation group acting on some set $\Omega_i$ which is clear from the context, then we will also use the notation $\U_M(\F)$ for $\U_\Delta(\F)$, where $\Delta$ is then the unique right-angled building of type $M$ over $I$ with parameters $(\lvert \Omega_i \rvert)_{i \in I}$.
\end{enumerate}
\end{remark}

The universal groups come equipped with a natural topology, namely the \emph{permutation topology}, which is defined by taking as an identity neighborhood basis the collection of all pointwise stabilizers of finite subsets of $\Delta$.

\medskip

The following observation is worth mentioning, because this is precisely the type of result we will be generalizing later.
\begin{lemma}
	\label{lem:universalreducible}
	Let $\Delta$ be a reducible right-angled building $\Delta$ over $I$.
	Let $J_1,\dots, J_m$ be the connected components of the diagram of $\Delta$. Then the universal group $\U_\Delta(\F)$ splits as~a direct product
	\[\U_\Delta(\F) \cong \U_{\R_1}\!\left(\textstyle\restrict\F{J_1}\right) \times \dotsm \times \U_{\R_m}\!\left(\textstyle\restrict\F{J_m}\right)\!,\]
	where each $\R_\ell$ is a residue of type $J_\ell$.
\end{lemma}
\begin{proof}
	Since $\Delta$ is isomorphic to the direct product $\R_1 \times \dotsm \times \R_m$ and has automorphism group $\Aut(\Delta) \cong \Aut(\R_1) \times \dotsm \times \Aut(\R_m)$, this follows immediately from the definition.
\end{proof}

\section{City products}
\label{sec:cityproduct}

In this section, we develop a construction for creating new right-angled buildings of a higher rank by gluing together lower rank buildings along another diagram. Our construction is inspired by the observation that the large-scale geometry of certain right-angled buildings (such as \cref{fig:coxeterrab}) resembles that of a tree; the city product structure explains this behavior in a broad sense.

\subsection{Weak homotopies}

We start with some combinatorics, the goal of which will become clear later on.
\begin{definition}
	\label{def:weakhomotopy}
	Let $i,j\in I$ with $m_{ij}=2$ and define the set
	\[P(i,j) = \bigl\{w\in \{i,j\}^* \bigm| \text{$w$ contains at least one $i$ and at least one $j$}\bigr\}.\]
	A \emph{weak homotopy} is a transformation of a word $w_1\,p\,w_2$ into a word $w_1\,p'\,w_2$ where $w_1,w_2\in I^*$ and $p,p'\in P(i,j)$. Two words $w$ and $w'$ are \emph{weakly homotopic} if $w$ can be transformed into $w'$ by a sequence of weak homotopies.
\end{definition}

\begin{definition}
	Let $\prec$ be a total order on $I$. Endow $I^*$ with the induced lexico\-graphical order. Then every word $w\in I^*$ is homotopic to a unique lexicographically minimal word that we call the \emph{normal form} of $w$.
\end{definition}

\begin{proposition}
	\label{prop:normalform}
    Let $(W,S)$ be a right-angled Coxeter system over $I$ and let $\prec$ be a total order on $I$.
	\begin{enumerate}
		\item If two words are homotopic, then their normal forms are equal.
		\item\label{prop:normalform:2} A word is reduced if and only if its normal form contains no consecutive duplicate letters.
		\item\label{prop:normalform:3} The normal forms of weakly homotopic words are equal up to consecutive duplicate letters.
	\end{enumerate}
\end{proposition}
\begin{proof}
	Claim (i) follows immediately from the definitions. For (ii), let $w\simeq w_1\,ii\,w_2$ and assume that the normal form contains no subword $ii$. Mark the two letters $i$ in $w_1\,ii\,w_2$ and write the normal form as $n_1\,i\,n_2\,i\,n_3$ (where the two letters $i$ are the marked ones). Then $n_2$ is not the empty word, so let $k$ be its first letter; by assumption, $k\neq i$. By homotopy, all letters in $n_2$ are contained in $\{i\} \cup \{i\}^\perp$, where $\{i\}^\perp := \{ j \in I \mid m_{ij} = 2 \}$. It follows that the normal form cannot be lexicographically minimal: if $i\prec k$, then the homotopic word $n_1\,ii\,n_2\,n_3$ is lexicographically smaller, and if $i\succ k$, then $n_1\,n_2\,ii\,n_3$ is smaller. Claim (ii) follows. For claim (iii), it suffices to note that the effect of a weak homotopy of a word on its normal form is that a subword $i^mj^n$ with $m\geq 1$, $n\geq 1$, is replaced by another such word.
\end{proof}

\begin{corollary}
	\label{cor:weaklyhomotopicreduced}
    Let $(W,S)$ be a right-angled Coxeter system over $I$ and let $\prec$ be a total order on $I$.
	If two reduced words $w,w'\in I^*$ are weakly homotopic, then they are homotopic.
\end{corollary}
\begin{proof}
	Letting $\prec$ be any total order, this follows readily from \cref{prop:normalform}\ref{prop:normalform:2} and \ref{prop:normalform:3}.
\end{proof}

\subsection{City product of diagrams}

Now let us go back to the building realm and define an operation on the diagrams first.

\begin{definition}
	Let $M$ be a diagram of rank $n$ over the index set $\{1,\dots,n\}$, and for each $\ell \in \{ 1,\dots,n \}$, let $M_\ell$ be a diagram over an index set $I_\ell$. Then we define a new diagram as follows:
	\begin{enumerate}
		\item the index set is the disjoint union $I = \bigsqcup_{\ell=1}^n I_\ell$; 
		\item for every pair of elements $i\in I_\ell$ and $i'\in I_{\ell'}$, we set
		\[ m_{ii'} := \begin{cases}
		    m_{ii'} \text{ (considered in $M_\ell$)} & \text{ if } \ell = \ell' ; \\
		    m_{\ell\ell'} \text{ (considered in $M$)} & \text{ if } \ell \neq \ell' .
		\end{cases} \]
	\end{enumerate}
	We call this the \emph{city product} of the diagrams $M_1,\dots,M_n$ over $M$ and denote it by $\cityprod_{M}(M_1,\dots,M_n)$.	Clearly its rank is $\sum_{\ell=1}^n |I_\ell|$.
\end{definition}

Notice that the special case of a city product over an edgeless diagram (i.e., $m_{ij}=2$ for all $1\leq i\neq j\leq n$) results in nothing more than the disjoint union of the diagrams $M_1,\dots,M_n$. Two more examples are given in \cref{fig:cityprod}. (More examples will occur in \cref{ex:isom} later.) Our choice for the symbol $\cityprod$ for city products is inspired by the example from \cref{fig:cityprod:B}.

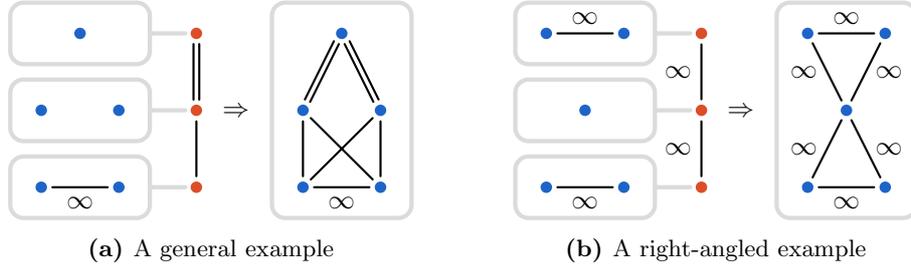
\begin{figure}[ht]
	\centering
	\begin{subfigure}[b]{.46\textwidth}
		\centering
		\begin{tikzpicture}[scale=.85]
			\path (0,0) node[myvertex,ugentblue] (A1) {}
				+(1,0) node[myvertex,ugentblue] (A2) {}
				++(0,1) node[myvertex,ugentblue] (A3) {}
				+(1,0) node[myvertex,ugentblue] (A4) {}
				++(.5,1) node[myvertex,ugentblue] (A5) {}
				(2,0) node[myvertex,ugentred] (B1) {}
				++(0,1) node[myvertex,ugentred] (B2) {}
				++(0,1) node[myvertex,ugentred] (B3) {};
			\draw[myedge,ultra thick,mydarkgray,rounded corners=5pt]
				(-.4,-.4) rectangle (1.4,.4) (1.4,0) -- (B1)
				(-.4,.6) rectangle (1.4,1.4) (1.4,1) -- (B2)
				(-.4,1.6) rectangle (1.4,2.4) (1.4,2) -- (B3);
			\draw[myedge] (A1) -- node[below] {$\infty$} (A2);
			\draw[myedge] (B1) -- (B2) (B2.105) -- (B3.255) (B2.75) -- (B3.285);
		\end{tikzpicture}
		\hspace*{1ex}$\Rightarrow$\hspace*{1ex}
		\begin{tikzpicture}[scale=.85]
			\path (0,0) node[myvertex,ugentblue] (A1) {}
				+(1,0) node[myvertex,ugentblue] (A2) {}
				++(0,1) node[myvertex,ugentblue] (A3) {}
				+(1,0) node[myvertex,ugentblue] (A4) {}
				++(.5,1) node[myvertex,ugentblue] (A5) {};
			\draw[myedge] (A1) -- node[below] {$\infty$} (A2) -- (A3) -- (A1) -- (A4) -- (A2)
			(A3.80) -- (A5.225) (A3.50) -- (A5.255) (A4.130) -- (A5.285) (A4.100) -- (A5.315);
			\draw[myedge,ultra thick,mydarkgray,rounded corners=5pt] (-.4,-.4) rectangle (1.4,2.4);
		\end{tikzpicture}
		\caption{A general example}
	\end{subfigure}
	\qquad
	\begin{subfigure}[b]{.46\textwidth}
		\centering
		\begin{tikzpicture}[scale=.85]
			\path (0,0) node[myvertex,ugentblue] (A1) {}
				+(1,0) node[myvertex,ugentblue] (A2) {}
				+(.5,1) node[myvertex,ugentblue] (A5) {}
				+(0,2) node[myvertex,ugentblue] (A3) {}
				+(1,2) node[myvertex,ugentblue] (A4) {}
				(2,0) node[myvertex,ugentred] (B1) {}
				++(0,1) node[myvertex,ugentred] (B2) {}
				++(0,1) node[myvertex,ugentred] (B3) {};
			\draw[myedge,ultra thick,mydarkgray,rounded corners=5pt]
				(-.4,-.4) rectangle (1.4,.4) (1.4,0) -- (B1)
				(-.4,.6) rectangle (1.4,1.4) (1.4,1) -- (B2)
				(-.4,1.6) rectangle (1.4,2.4) (1.4,2) -- (B3);
			\draw[myedge] (A1) -- node[below] {$\infty$} (A2);
			\draw[myedge] (A3) -- node[above] {$\infty$} (A4);
			\draw[myedge] (B1) -- node[left] {$\infty$} (B2) -- node[left] {$\infty$} (B3);
		\end{tikzpicture}
		\hspace*{1ex}$\Rightarrow$\hspace*{1ex}
		\begin{tikzpicture}[scale=.85]
			\path (0,0) node[myvertex,ugentblue] (A1) {}
				+(1,0) node[myvertex,ugentblue] (A2) {}
				+(.5,1) node[myvertex,ugentblue] (A5) {}
				+(0,2) node[myvertex,ugentblue] (A3) {}
				+(1,2) node[myvertex,ugentblue] (A4) {};
			\draw[myedge,ultra thick,mydarkgray,rounded corners=5pt] (-.4,-.4) rectangle (1.4,2.4);
			\draw[myedge] (A1) -- node[left] {$\infty$} (A5) -- node[right] {$\infty$} (A2);
			\draw[myedge] (A3) -- node[left] {$\infty$} (A5) -- node[right] {$\infty$} (A4);
			\draw[myedge] (A1) -- node[below] {$\infty$} (A2);
			\draw[myedge] (A3) -- node[above] {$\infty$} (A4);
		\end{tikzpicture}
		\caption{A right-angled example}
        \label{fig:cityprod:B}
	\end{subfigure}
\caption{City products of diagrams}
\label{fig:cityprod}
\end{figure}

\begin{lemma}
	 The diagram $\cityprod_{M}(M_1,\dots,M_n)$ with $n\geq 2$ is irreducible if and only if $M$ is irreducible.
\end{lemma}
\begin{proof}
	This follows immediately from the definition.
\end{proof}

\subsection{City product of right-angled buildings}

We can now continue to define city products of right-angled buildings.
\begin{definition}
	Let $M$ be a right-angled diagram over the index set $\{1,\dots,n\}$, and for each $\ell \in \{ 1,\dots,n \}$, let $\Delta_\ell$ be a semiregular right-angled building of type $M_\ell$ over $I_\ell$. Then we define the \emph{city product} of the buildings $\{\Delta_1,\dots,\Delta_n\}$ over $M$ as follows:
	\begin{enumerate}
		\item the index set is the disjoint union $I = \bigsqcup_{\ell=1}^n I_\ell$; 
		\item the (right-angled) diagram is the city product of diagrams $\cityprod_{M}(M_1,\dots,M_n)$;
		\item for each $i\in I$, the parameter $q_i$ of the new building is the parameter $q_i$ of $\Delta_\ell$, where $i\in I_\ell$.
	\end{enumerate}
	By \cref{thm:rabsexist}, this defines a unique semiregular right-angled building (up to isomorphism), that we denote by $\cityprod_M(\Delta_1,\dots,\Delta_n)$.
	It will be convenient to define $\ell(i)$ (for $i \in I$) as the unique number in $\{1,\dots,n\}$ such that $i\in I_{\ell(i)}$.
\end{definition}

Note that for each $\ell \in \{ 1,\dots,n \}$, the residues of type $I_\ell\subseteq I$ of the city product $\cityprod_M(\Delta_1,\dots,\Delta_n)$ are isomorphic to the original building $\Delta_\ell$. As a special case of \cref{cor:residuecoloring}, we then obtain:

\begin{lemma}
	\label{lem:cityproductresidue}
	Let $\Delta = \cityprod_M(\Delta_1,\dots,\Delta_n)$ be a city product. Then for each $\ell \in \{ 1,\dots,n \}$, there is a map $\varphi_\ell\colon \Delta\to\Delta_\ell$ with the following properties:
	\begin{enumerate}
		\item for each residue $\R$ of type $I_\ell$, the restriction $\restrict{\varphi_\ell}{\R}\colon\R\to\Delta_\ell$ is an isomorphism;
		\item for each residue $\R$ of type $I\setminus I_\ell$, the restriction $\restrict{\varphi_\ell}{\R}\colon\R\to\Delta_\ell$ is a constant map.
	\end{enumerate}
\end{lemma}
\begin{proof}
	This follows immediately from \cref{cor:residuecoloring}.
\end{proof}

We can then easily lift colorings of the subbuildings to a coloring of the full city product.

\begin{lemma}
	\label{lem:cityproductcoloring}
	Let $\Delta = \cityprod_M(\Delta_1,\dots,\Delta_n)$ be a city product. For each $\ell \in \{ 1,\dots,n \}$, let $\lambda^\ell$ be a legal coloring of $\Delta_\ell$ with color sets $\Omega_i$ (where $i$ ranges over $I_\ell$). Then the collection of maps
	\[\lambda'_i = \lambda^{\ell(i)}_i\circ\varphi_{\ell(i)}\]
	provides a legal coloring of $\Delta$ with color sets $\Omega_i$ (where $i$ ranges over $I = \bigsqcup_{\ell=1}^n I_\ell$).
\end{lemma}
\begin{proof}
	This follows immediately from \cref{lem:cityproductresidue} and the definition of legal colorings.
\end{proof}

The city product construction over a diagram $M$ essentially glues together smaller rank buildings as if they were chambers of a building of type $M$, hence the fact that the original buildings reemerge locally as residues (\cref{lem:cityproductresidue}) should not be surprising. However, we can also recover a building of type $M$ at the global scale by relaxing the adjacencies.
\begin{definition}
	Let $\Delta = \cityprod_M(\Delta_1,\dots,\Delta_n)$ be a city product, where $M$ is a right-angled diagram over $\{1,\dots,n\}$. The \emph{skeletal building} of $\Delta$ is the chamber system $\Phi$ over the index set $\{1,\dots,n\}$ with the same chamber set as $\Delta$, but with coarser adjacencies: we declare two chambers $c,d\in\Delta$ to be $\ell$-adjacent in $\Phi$ if and only if they lie in the same residue of type $I_\ell$ in $\Delta$.
\end{definition}
We will prove in \cref{prop:cityproductskeletalstuff} that the skeletal building of a city product is, in fact, a building. First, we need an auxiliary definition and some combinatorial lemmas, laying the bridge between city products and weak homotopies.
\begin{definition}\label{def:parkour}
    Let $I = \bigsqcup_{\ell=1}^n I_\ell$.
    \begin{enumerate}
        \item 
        	The \emph{parkour map} of $(I_1,\dots,I_n)$ is the map
        	\[r\colon I^* \to \{1,\dots,n\}^*\]
        	that first replaces every letter $i\in I$ by $\ell(i)\in\{1,\dots,n\}$ and then removes consecutive duplicates (i.e., replaces them by a single letter).
        \item\label{def:parkour:bar}
            The map $\{1,\dots,n\}^* \to \{1,\dots,n\}^*$ that replaces consecutive duplicates by a single letter will be denoted by $v \mapsto \overline{v}$.
        \item
            The maximal subwords of a word $w\in I^*$ with letters in a common subset $I_\ell$ are called the \emph{blocks} of $w$. These are precisely the maximal subwords such that the image under~$r$ is a single letter.
    \end{enumerate}
\end{definition}

\begin{example}
    Consider the index sets
    \[I_1 = \{1_a, 1_b, 1_c\},
    		\quad I_2 = \{2_a, 2_b\},
    		\quad I_3 = \{3_a, 3_b, 3_c\},
    		\quad I = I_1\cup I_2\cup I_3.\]
    Then for the word $w = 2_a\, 2_b\, 3_c\, 1_c\, 1_a\, 1_b\, 1_a\, 3_b$, the image is $r(w) = 2313$. The blocks are the words
    \[2_a\, 2_b,\quad 3_c,\quad 1_c\, 1_a\, 1_b\, 1_a,\quad 3_b.\]
\end{example}

The interpretation in terms of the skeletal building is now clear: Let $\Delta = \cityprod_M(\Delta_1,\dots,\Delta_n)$ be a city product of type $I = \bigsqcup_{\ell=1}^n I_\ell$ and let $\Phi$ be its skeletal building.
If $w$ is the type of a gallery in $\Delta$, then $r(w)$ is the type of a gallery in $\Phi$ with the same extremities, but replacing subgalleries in residues of type $I_\ell$ by a single jump of type $\ell$. (This behavior explains our choice for the terminology ``parkour map''.)

\medskip

When viewed as elements of the corresponding Coxeter groups, the interplay between words in $I^*$ and words in $\{1,\dots,n\}^*$ is not completely trivial --- especially when considering reduced words. As illustrated in \cref{fig:parkourmap}, images of reduced words under the parkour map are not necessarily reduced, nor are images of equivalent words necessarily equivalent.

\begin{figure}[!ht]
\newcommand{\drawgrid}{%
	\foreach\i in {0,...,4}
		\foreach\j in {0,...,3}
			\node[myvertex] (P\i\j) at (\i,\j) {};
	\foreach\i in {0,...,4}
		{\draw[myedge,mydarkgray] (P\i0) -- (P\i1) -- (P\i2) -- (P\i3);
		\draw[myedge,mydarkgray,-dots] (P\i0) -- (\i,-.3);
		\draw[myedge,mydarkgray,-dots] (P\i3) -- (\i,3.3);
		}
	\foreach\j in {0,...,3}
		{\draw[myedge,mydarkgray] (P0\j) -- (P1\j) -- (P2\j) -- (P3\j) -- (P4\j);
		\draw[myedge,mydarkgray,-dots] (P0\j) -- (-.3,\j);
		\draw[myedge,mydarkgray,-dots] (P4\j) -- (4.3,\j);
		}
}
	\centering
	\begin{subfigure}{\textwidth}
		\centering
		\begin{tikzpicture}
			\path (0,0) node[myvertex,ugentblue] (A1) {}
				+(1,0) node[myvertex,ugentblue] (A2) {}
				++(0,1) node[myvertex,ugentblue] (B1) {}
				+(1,0) node[myvertex,ugentblue] (B2) {}
				(2,0) node[myvertex,ugentred] (X1) {}
				++(0,1) node[myvertex,ugentred] (X2) {};
			\draw[myedge,ultra thick,mydarkgray,rounded corners=5pt]
				(-.4,-.4) rectangle (1.4,.4) (1.4,0) -- (X1)
				(-.4,.6) rectangle (1.4,1.4) (1.4,1) -- (X2);
			\draw[myedge] (A1) -- node[above] {$\infty$} (A2);
			\draw[myedge] (B1) -- node[above] {$\infty$} (B2);
		\end{tikzpicture}
		\quad$\Rightarrow$\quad
		\begin{tikzpicture}
			\path (0,0) node[myvertex,ugentblue,label=left:$2_a$] (A1) {}
				+(1,0) node[myvertex,ugentblue,label=right:$2_b$] (A2) {}
				++(0,1) node[myvertex,ugentblue,label=left:$1_a$] (B1) {}
				+(1,0) node[myvertex,ugentblue,label=right:$1_b$] (B2) {};
			\draw[myedge,ultra thick,mydarkgray,rounded corners=5pt] (-.6,-.4) rectangle (1.6,1.4);
			\draw[myedge] (A1) -- node[above] {$\infty$} (A2) (B1) -- node[above] {$\infty$} (B2);
		\end{tikzpicture}
		\caption{The ambient (reducible) city product. Notice that this occurs as a residue of the irreducible city product from \cref{fig:cityprod:B}.}
	\end{subfigure}
	\par
	\begin{subfigure}{\textwidth}
	\begin{subfigure}[b]{.32\textwidth}
		\centering
		\begin{tikzpicture}[x=7mm,y=7mm]
			\drawgrid
			\draw[myedge,black] (P03) -- (P13) -- (P23) -- (P33) -- (P43) -- (P42) -- (P41) -- (P40);
		\end{tikzpicture}
		$r(w) = 12$,
	\end{subfigure}
	\begin{subfigure}[b]{.32\textwidth}
		\centering
		\begin{tikzpicture}[x=7mm,y=7mm]
			\drawgrid
			\draw[myedge,black] (P03) -- (P13) -- (P23) -- (P22) -- (P21) -- (P20) -- (P30) -- (P40);
		\end{tikzpicture}
		$r(w) = 121$,
	\end{subfigure}
	\begin{subfigure}[b]{.32\textwidth}
		\centering
		\begin{tikzpicture}[x=7mm,y=7mm]
			\drawgrid
			\draw[myedge,black] (P03) -- (P02) -- (P12) -- (P11) -- (P21) -- (P31) -- (P30) -- (P40);
		\end{tikzpicture}
		$r(w) = 212121$.
	\end{subfigure}
    \caption{The parkour map $r \colon \{ 1_a, 1_b, 2_a, 2_b \}^* \to \{ 1, 2 \}^*$. Horizontal lines alternate between $1_a$~and~$1_b$, vertical lines alternate between $2_a$ and $2_b$.}
	\end{subfigure}
\caption{The effect of the parkour map on equivalent types of minimal galleries}
\label{fig:parkourmap}
\end{figure}

The following slightly technical lemma explains the connection in more detail.
\begin{lemma}
	\label{lem:parkourmap}
	Let $M$ be a diagram of rank $n$ over the index set $\{1,\dots,n\}$ and for each $\ell \in \{ 1,\dots,n \}$, let $M_\ell$ be a diagram over $I_\ell$.
	Consider the city product $\cityprod_{M}(M_1,\dots,M_n)$, with index set $I = \bigsqcup_{\ell=1}^n I_\ell$.
	Let $u\in I^*$ and let $r\colon I^*\to\{1,\dots,n\}^*$ be the parkour map.
	\begin{enumerate}
		\item\label{lem:parkourmap:1} If $u\simeq u'$, then $r(u)$ and $r(u')$ are weakly homotopic (in the sense of \cref{def:weakhomotopy}).
		\item\label{lem:parkourmap:2} Assume that we have a homotopy $r(u)\simeq v$.
            Then there exists $u'\in I^*$ such that $u'\simeq u$ and $r(u')=\overline{v}$,
            where $\overline{v}$ is as in \itemref{def:parkour}{bar}.
			\[\begin{tikzcd}[dims={5em}{4.5em}]
				u \ar[d,"r"] \ar[rr,"\simeq"] && u' \ar[d,"r"] \\
				r(u) \ar[r,"\simeq"] & v \ar[r,dashed] & \overline{v}
			\end{tikzcd}\]
		\item\label{lem:parkourmap:3} If\, $u$ is reduced, then all blocks of $u$ are reduced.
		\item\label{lem:parkourmap:4} If all blocks of $u$ are reduced and $r(u)$ is reduced, then $u$ is reduced.
		\item\label{lem:parkourmap:5} If\, $u\simeq u'$ and both $r(u)$ and $r(u')$ are reduced, then $r(u)\simeq r(u')$.
	\end{enumerate}
\end{lemma}
\begin{proof}
    \begin{enumerate}
        \item 
        	Consider an elementary homotopy $u=u_1\,ij\,u_2 \simeq u_1\,ji\,u_2$. If $\ell(i)=\ell(j)$, then the image under $r$ remains unchanged. Assume now that $\ell(i)\neq\ell(j)$. We distinguish three cases for the subword $u_1$ of $u$:
        	\begin{description}\setlength{\itemindent}{-7ex}
        		\item[{[\textcolor{ugentblue}{\textsf{L.a}}]}] $u_1$ is nonempty and the last letter of $u_1$ is in $I_{\ell(i)}$,
        		\item[{[\textcolor{ugentred}{\textsf{L.b}}]}] $u_1$ is nonempty and the last letter of $u_1$ is in $I_{\ell(j)}$,
        		\item[{[\textsf{L.c}]}] $u_1$ is the empty word, or the last letter of $u_1$ is neither in $I_{\ell(i)}$ nor $I_{\ell(j)}$.
        	\end{description}
        	Analogously, we distinguish three cases for $u_2$:
        	\begin{description}\setlength{\itemindent}{-7ex}
        		\item[{[\textcolor{ugentblue}{\textsf{R.a}}]}] $u_2$ is nonempty and the first letter of $u_2$ is in $I_{\ell(i)}$,
        		\item[{[\textcolor{ugentred}{\textsf{R.b}}]}] $u_2$ is nonempty and the first letter of $u_2$ is in $I_{\ell(j)}$,
        		\item[{[\textsf{R.c}]}] $u_2$ is the empty word, or the first letter of $u_2$ is neither in $I_{\ell(i)}$ nor $I_{\ell(j)}$.
        	\end{description}
        	\smallskip
        	Depending on the nine combinations of possibilities, the elementary homotopy $u_1\,ij\,u_2 \simeq u_1\,ji\,u_2$ transforms the image $r(u)$ by substituting some subword in $\{ij,ji,iji,jij,ijij,jiji\}$ (where we have simply written $i$ and $j$ instead of $\ell(i)$ and $\ell(j)$ for better readability) into another such word; see \cref{fig:parkourhomotopy}. In each case, the result is weakly homotopic to $r(u)$.
        	
            \begin{table}[!ht]
            	\centering
            	\begin{tikzpicture}
            		\path (0,0) node[myvertex] (L) {}
            			+(45:1) node[myvertex] (U) {}
            			++(-45:1) node[myvertex] (D) {}
            			++(45:1) node[myvertex] (R) {};
            		\path (D.center) to node[midway,rotate=-90] {\strut$\smash{\Longrightarrow}$} (U.center);
            		\draw[myedge,ugentblue] (L) -- node[above left] {$i$\vphantom{$j$}} (U);
            		\draw[myedge,ugentblue,dashed] (D) -- node[below right] {$i$} (R);
            		\draw[myedge,ugentblue,-dots] (L) -- +(225:.667) node[pos=1.7,left] {\bfseries\sffamily\textcolor{ugentblue}{L.a}};
            		\draw[myedge,ugentblue,-dots] (R) -- +(45:.667) node[pos=1.7,right] {\bfseries\sffamily\textcolor{ugentblue}{R.a}};
            		\draw[myedge,ugentred,dashed] (L) -- node[below left] {$j$} (D);
            		\draw[myedge,ugentred] (U) -- node[above right] {$j$} (R);
            		\draw[myedge,ugentred,-dots] (L) -- +(135:.667) node[pos=1.7,left] {\bfseries\sffamily\textcolor{ugentred}{L.b}};
            		\draw[myedge,ugentred,-dots] (R) -- +(-45:.667) node[pos=1.7,right] {\bfseries\sffamily\textcolor{ugentred}{R.b}};
            		\draw[myedge,-dots] (L) -- +(180:.667) node[pos=2.3] {\bfseries\sffamily L.c};
            		\draw[myedge,-dots] (R) -- +(0:.667) node[pos=2.3] {\bfseries\sffamily R.c};;
            	\end{tikzpicture}
            	\par\bigskip
            	\renewcommand{\arraystretch}{1.25}
            	\newcommand{\blah}[2]{$\mathllap{#1}\rightsquigarrow\mathrlap{#2}$}
            	\begin{tabular}{p{6mm}*{3}{>{\centering\arraybackslash}p{20mm}}}
            		\toprule
            		& {\bfseries\sffamily\textcolor{ugentblue}{R.a}} & {\bfseries\sffamily\textcolor{ugentred}{R.b}} & {\bfseries\sffamily {R.c}}\\
            		\cmidrule(l){2-4}
            		{\bfseries\sffamily\textcolor{ugentblue}{L.a}} & \blah{iji}{iji} & \blah{ij}{ijij} & \blah{ij}{iji} \\
            		{\bfseries\sffamily\textcolor{ugentred}{L.b}} & \blah{jiji}{ji} & \blah{jij}{jij} & \blah{jij}{ji} \\
            		{\bfseries\sffamily L.c} & \blah{iji}{ji} & \blah{ij}{jij} & \blah{ij}{ji} \\
            		\bottomrule
            	\end{tabular}
            	\bigskip
                \caption{The effect of an elementary homotopy on the image of the parkour map. We have simply written $i$ and $j$ instead of $\ell(i)$ and $\ell(j)$ for better readability.}
                \label{fig:parkourhomotopy}
            \end{table}
            \vspace*{-3ex}
    	\item
        	Consider an elementary homotopy $r(u) = v_1\,\ell_1\ell_2\,v_2 \simeq v_1\,\ell_2\ell_1\,v_2 = v$ with $1\leq\ell_1\neq \ell_2\leq n$ and such that $\ell_1$ and $\ell_2$ commute in $M$. Then we can write $u = u_1\,b_1b_2\,u_2$, where $b_1$ and $b_2$ are the blocks corresponding to $\ell_1$ and~$\ell_2$, respectively. Since $b_1$ and $b_2$ have only letters in $I_{\ell_1}$ and in $I_{\ell_2}$, which are sets of pairwise commuting generators in the Coxeter system, we have a homotopy $u' = u_1\,b_2b_1\,u_2 \simeq u_1\,b_1b_2\,u_2$ satisfying $r(u') = \overline{v}$.
        	The claim now follows by induction on the number of elementary homotopies needed to go from $r(u)$ to~$v$.
    	\item
        	This is obvious since any subword of a reduced word is reduced.
    	\item
        	Assume by means of contraposition that every block of $u$ is reduced while $u$ is not, i.e., there is a homotopy $u\simeq w_1\,ii\,w_2$. Mark these two letters $i$ and let $b_1$ and $b_2$ be the two blocks of $u$ containing these two marked letters. Since every block is reduced, we have $b_1\neq b_2$, hence $u = u_1\,b_1\,u_2\,b_2\,u_3$ such that $u_2$ is nonempty and $m_{ij}=2$ for every letter $j$ in $u_2$. The image then satisfies $r(u) = r(u_1)\,\ell(i)\,r(u_2)\,\ell(i)\,r(u_3)$. By construction, $\ell(i)$ commutes with every letter in $r(u_2)$. Hence $r(u)$ is not reduced.
    	\item
        	This follows from \ref{lem:parkourmap:1} and \cref{cor:weaklyhomotopicreduced}.
        \qedhere
    \end{enumerate}
\end{proof}

\begin{proposition}
	\label{prop:cityproductskeletalstuff}
	Let $M$ be a diagram of rank $n$, let $\Delta_1,\dots,\Delta_n$ be right-angled buildings,
	let $\Delta :=\cityprod_M(\Delta_1,\dots,\Delta_n)$ be their city product and let $\Phi$ be its skeletal building.
	Then
	\begin{enumerate}
		\item $\Phi$ is a right-angled building of type $M$ over $\{1,\dots,n\}$;
		\item $\Phi$ is semiregular with parameters $q_\ell = \lvert\Delta_\ell\rvert$ for every $\ell \in \{ 1,\dots,n \}$;
		\item $\ell$-panels of\, $\Phi$ (as sets of chambers) are $I_\ell$-residues of $\Delta$ and vice versa;
		\item the maps $\varphi_\ell$ with $\ell \in \{ 1,\dots,n \}$ (introduced in \cref{lem:cityproductresidue}) provide a legal coloring of\, $\Phi$ with color sets $\Delta_\ell$.
	\end{enumerate}
\end{proposition}
\begin{proof}
	The only nontrivial claim is that $\Phi$ is indeed a building of type $M$; the other claims will then follow immediately from the definitions. Let us first write out the Weyl distance function in $\Phi$ and then verify that it satisfies the necessary properties. 
	
	Denote by $W_\Delta$ the Weyl group of the building $\Delta$. Recall from \cref{def:evaluationmorphism} the evaluation morphism $\epsilon_\Delta\colon I^* \to W_\Delta$. On the other hand, we have a Coxeter group $W_\Phi$ of type $M$, together with an evaluation morphism $\epsilon_\Phi\colon \{1,\dots,n\}^* \to W_\Phi$. Next, we let $s\colon W_\Delta\to I^*$ be a section of $\epsilon_\Delta$ with reduced images such that for each $w\in W_\Delta$, the word length $\lvert r(s(w))\rvert$ is minimal among all possible choices for $s(w)$.
	We claim that for each $w \in W_\Delta$, the word $r(s(w))$ in $\{ 1,\dots,n \}^*$ is reduced.
	Indeed, if $r(s(w))$ would not be reduced, then by \itemref{lem:parkourmap}{2}, we would find some word $u' \in I^*$ with $u' \simeq s(w)$ (so that $\epsilon_\Delta(u') = \epsilon_\Delta(s(w))$) for which $r(u')$ has smaller word length than $r(s(w))$.
	
	Finally, we can define
	\[\delta_\Phi = \epsilon_\Phi\circ r\circ s\circ \delta_\Delta \colon \Phi\times\Phi\to W_\Phi.\]
	(Notice that the composition $\epsilon_\Phi\circ r\circ s$ is a map $W_\Delta\to W_\Phi$ but by no means a group homomorphism.)
	\[\begin{tikzcd}[dims={8em}{4.5em}]
		\Delta\times\Delta \ar[r,"\delta_\Delta"] \ar[d,equal] &
			W_\Delta \ar[d,dashed] \arrow[bend left=18]{r}{s}
			& I^* \ar[d,"r"] \arrow[bend left=18]{l}{\epsilon_\Delta} \\
		\Phi\times\Phi \ar[r,"\delta_\Phi"]
			& W_\Phi
			& \{1,\dots,n\}^* \ar[l,"\epsilon_\Phi",swap]
	\end{tikzcd}\]
	\smallskip
	
	Clearly, panels of $\Phi$ contain at least two chambers, since every such panel of $\Phi$ contains a panel of $\Delta$. Now consider a reduced word $v$ in $\{1,\dots,n\}^*$. Our goal is to show that $\delta_\Phi(c,d)=\epsilon_\Phi(v)$ if and only if there exists a gallery of type $v$ from $c$ to $d$ in $\Phi$.

	First, assume that $\delta_\Phi(c,d)=\epsilon_\Phi(v)$. By definition of $\delta_\Phi$, this means that the words $(r\circ s\circ\delta_\Delta)(c,d)$ and $v$ are equivalent. Moreover, both words are reduced, and hence homotopic by \itemref{thm:homotopicstuff}{2}. By \itemref{lem:parkourmap}{2}, this homotopy can be realized in $I^*$, i.e., we can find a word $u\in I^*$ such that $u \simeq (s\circ\delta_\Delta)(c,d)$ and $r(u) = v$. The homotopy $u \simeq (s\circ\delta_\Delta)(c,d)$ yields that $u$ is reduced, hence by the building axioms for $\Delta$, there is a minimal gallery in $\Delta$ of type $u$ from $c$ to $d$. Then $r(u)=v$ is the type of a gallery in $\Phi$ from $c$ to $d$.
	
	Conversely, assume that $\gamma$ is a gallery of type $v$ from $c$ to $d$ in $\Phi$. We can ``lift'' $\gamma$ to a gallery $\widebar\gamma$ in $\Delta$ with the same extremities, by replacing each $\ell$-adjacency in $\gamma$ by a minimal gallery in a residue of type $I_\ell$ of $\Delta$. Let $\widebar v$ be the type of $\widebar\gamma$. Note that $r(\widebar v)=v$ and that $\widebar v$ is reduced by \itemref{lem:parkourmap}{4}. Hence, we have  $\delta_\Delta(c,d)=\epsilon_\Delta(\widebar v)$, so that $s(\delta_\Delta(c,d))$ and $\widebar v$ are homotopic by \cref{thm:homotopicstuff}\ref{thm:homotopicstuff:1}~and~\ref{thm:homotopicstuff:2}. Then by \itemref{lem:parkourmap}{5}, the images $(r\circ s\circ\delta_\Delta)(c,d)$ and $r(\widebar v)=v$ are homotopic, so that finally
	\[\delta_\Phi(c,d) = (\epsilon_\Phi\circ r\circ s\circ\delta_\Delta)(c,d) = \epsilon_\Phi(v).\]
	This concludes our proof that $\Phi$ is a right-angled building of type $M$.
\end{proof}

\begin{example}
    Consider the Coxeter complex from \cref{fig:coxeterrab}.
    This is a thin building of rank three, so each of the three parameters $q_1$, $q_2$ and $q_3$ is equal to $2$.
    We can now view this building as a city product of two thin buildings, one of rank $1$ and one of rank $2$:
    
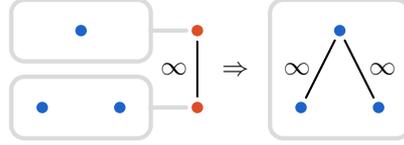
\[
		\begin{tikzpicture}[scale=.85]
			\path (0,0) node[myvertex,ugentblue] (A1) {}
				+(1,0) node[myvertex,ugentblue] (A2) {}
				+(.5,1) node[myvertex,ugentblue] (A5) {}
				(2,0) node[myvertex,ugentred] (B1) {}
				++(0,1) node[myvertex,ugentred] (B2) {};
			\draw[myedge,ultra thick,mydarkgray,rounded corners=5pt]
				(-.4,-.4) rectangle (1.4,.4) (1.4,0) -- (B1)
				(-.4,.6) rectangle (1.4,1.4) (1.4,1) -- (B2);
			\draw[myedge] (B1) -- node[left] {$\infty$} (B2);
		\end{tikzpicture}
		\hspace*{1ex} \Rightarrow \hspace*{1ex}
		\begin{tikzpicture}[scale=.85]
			\path (0,0) node[myvertex,ugentblue] (A1) {}
				+(1,0) node[myvertex,ugentblue] (A2) {}
				+(.5,1) node[myvertex,ugentblue] (A5) {};
			\draw[myedge,ultra thick,mydarkgray,rounded corners=5pt] (-.4,-.4) rectangle (1.4,1.4);
			\draw[myedge] (A1) -- node[left] {$\infty$} (A5) -- node[right] {$\infty$} (A2);
		\end{tikzpicture}
    \]
    The skeletal building $\Phi$ of this city product is now obtained by merging the adjacencies ``red'' and ``blue'' together to a new $\ell_1$-adjacency; the adjacency ``green'' is unchanged and is our new $\ell_2$-adjacency.
    By doing so, the residues of type $\{ \text{red}, \text{blue} \}$ in \cref{fig:coxeterrab} have now become panels (of type $\ell_1$), and as a result, $\Phi$ is now a building of rank two, which is no longer a thin building: it is a semiregular tree with parameters $q_1=4$ and $q_2=2$; see \cref{fig:skeletal}.
    (Notice that this is a tree viewed as a \emph{chamber complex}: the vertices of the corresponding tree correspond to the \emph{panels} of this chamber complex.)
\end{example}
\begin{figure}[!ht]
	\centering
	\vspace*{-3ex}
	\begin{tikzpicture}[scale=.9]
		\draw[myedge,ultra thick,mydarkgray,rounded corners=5pt,fill]
			(-.32,-.55) rectangle (1.3,.32);
		\path (0,0) node[myvertex,ugentblue,label=below:$1$] (A1) {}
			+(1,0) node[myvertex,ugentred,label=below:$3$] (A2) {}
			+(.5,1) node[myvertex,ugentgreen,label=right:$2$] (A5) {};
		\draw[myedge] (A1) -- node[left] {$\infty$} (A5) -- node[right] {$\infty$} (A2);
	\end{tikzpicture}
	\hspace{8ex}
	\begin{tikzpicture}[scale=.8]
		\pgfmathsetmacro{\s}{1+sqrt(2)/2}
		\draw[myedge,ultra thick,mydarkgray,rounded corners=5pt,fill]
			(-.7,-.7) rectangle (.7,.7)
			(\s-.7,\s-.7) rectangle (\s+.7,\s+.7)
			(-\s-.7,\s-.7) rectangle (-\s+.7,\s+.7)
			(\s-.7,-\s-.7) rectangle (\s+.7,-\s+.7)
			(-\s-.7,-\s-.7) rectangle (-\s+.7,-\s+.7);
		\foreach\P/\M in {(0,0)/C, (-\s,-\s)/LL, (-\s,\s)/UL, (\s,\s)/UR, (\s,-\s)/LR}
			\foreach\Q/\N in {(-.5,-.5)/LL, (-.5,.5)/UL, (.5,.5)/UR, (.5,-.5)/LR}
				\path \P +\Q node[myvertex] (\M-\N) {};
		\begin{scope}[myedge,ugentgreen,very thick]
			\draw[-dots] (LL-LR) -- ++(285:.4);
				\draw[-dots] (LL-LL) -- ++(225:.4);
				\draw[-dots] (LL-UL) -- ++(165:.4);
			\draw[-dots] (UL-LL) -- ++(195:.4);
				\draw[-dots] (UL-UL) -- ++(135:.4);
				\draw[-dots] (UL-UR) -- ++(75:.4);
			\draw[-dots] (UR-UL) -- ++(105:.4);
				\draw[-dots] (UR-UR) -- ++(45:.4);
				\draw[-dots] (UR-LR) -- ++(345:.4);
			\draw[-dots] (LR-UR) -- ++(15:.4);
				\draw[-dots] (LR-LR) -- ++(315:.4);
				\draw[-dots] (LR-LL) -- ++(255:.4);
			\draw (C-LL) -- (LL-UR) (C-UL) -- (UL-LR) (C-UR) -- (UR-LL) (C-LR) -- (LR-UL);				
		\end{scope}
		\draw[myedge,ugentred]
			\foreach\M in {C, LL, UL, UR, LR} {(\M-LL) -- (\M-LR) (\M-UL) -- (\M-UR)};
		\draw[myedge,ugentblue]
			\foreach\M in {C, LL, UL, UR, LR} {(\M-LL) -- (\M-UL) (\M-LR) -- (\M-UR)};
		\draw[myedge,gray,thin]
			\foreach\M in {C, LL, UL, UR, LR} {(\M-LL) -- (\M-UR) (\M-LR) -- (\M-UL)};
	\end{tikzpicture}
	
\caption{The skeletal building of the right-angled Coxeter complex from \cref{fig:coxeterrab} viewed as city product. The grey blocks are panels of size 4, the green lines are panels of size 2.}
\label{fig:skeletal}
\end{figure}
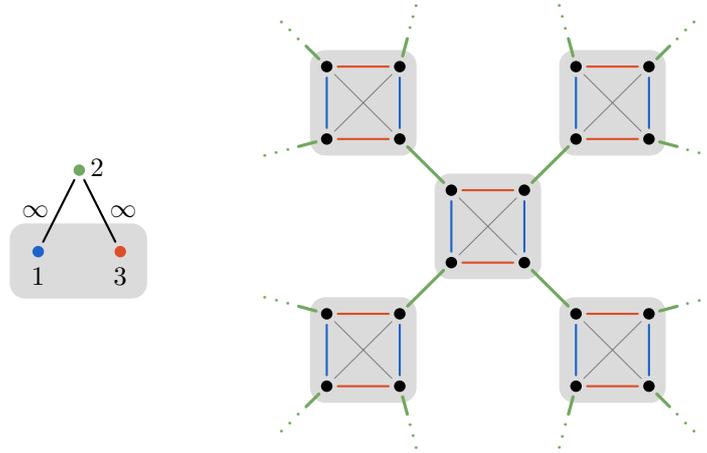

\begin{remark}
    The city product of right-angled diagrams can be interpreted in a purely graph-theoretical way, as follows.
    A subset $X$ of the vertex set of a graph is called a \emph{module} if it has the property that every vertex $v\notin X$ is either adjacent to all vertices in $X$ or adjacent to no vertex in $X$.
    A graph is called \emph{prime} if it has no non-trivial modules.
    
    A right-angled diagram is then a non-trivial city product of lower rank right-angled diagrams if and only if its underlying graph is not prime.
    A decomposition of a right-angled diagram as a non-trivial city product of lower rank right-angled diagrams then corresponds to a \emph{modular partition} of the underlying graph.
\end{remark}


\section{Universal groups of city products}

Earlier in \cref{lem:universalreducible}, we already observed that the universal group construction behaves nicely with respect to disjoint unions of diagrams. This operation on diagrams is a special case of the city product from \cref{sec:cityproduct} (over a diagram with only isolated nodes). In this section, we generalize \cref{lem:universalreducible} to arbitrary city products.
More precisely, we will show that the universal group over a city product of buildings $\cityprod_M(\Delta_1,\dots,\Delta_n)$ is isomorphic to the universal group over the skeletal building of the universal groups over the buildings $\Delta_i$.


\begin{theorem}
	\label{thm:universalcityproduct}
	Let $M$ be a right-angled diagram of rank $n$. For each $\ell \in \{ 1,\dots,n \}$, let $\Delta_\ell$ be a semiregular right-angled building of type $M_\ell$ over $I_\ell$, equipped with a legal coloring $\lambda^\ell$ with color sets $\Omega_i$ (indexed by $i \in I_\ell$).
	Let $\Delta :=\cityprod_M(\Delta_1,\dots,\Delta_n)$ be their city product over $I = \bigsqcup_{\ell=1}^n I_\ell$ and let $\Phi$ be its skeletal building over $\{ 1,\dots,n \}$.
	
	Assume that for each $i \in I$, we have a permutation group $F_i \leq \Sym(\Omega_i)$, giving rise to local data $\F$ over $I$.
	These local data restrict to local data $\F_\ell$ over $I_\ell$ for each $\ell \in \{ 1,\dots,n \}$.
	

    Then we have an isomorphism of topological groups
	\[ \U_\Delta(\F) \cong \U_\Phi\Bigl( \bigl( \U_{\Delta_\ell}(\F_\ell) \bigr)_{\ell \in \{ 1,\dots,n \}} \Bigr) .\]
\end{theorem}
\begin{proof}
	Equip $\Delta$ with the coloring $\lambda'$ from \cref{lem:cityproductcoloring}, assigning colors in the sets $\Omega_i$ (indexed by $i\in I$). Also equip its skeletal building $\Phi$ with the coloring $\varphi$ from \cref{prop:cityproductskeletalstuff}, assigning colors in the sets $\Delta_\ell$ (indexed by $\ell\in\{1,\dots,n\}$).
	Let $\F'$ be the local data over $\{ 1,\dots,n \}$ defined by
	\[ F'_\ell := \U_{\Delta_\ell}(\F_\ell) \quad \text{for each } \ell \in \{ 1,\dots,n \} . \]
	
	First, every automorphism of $\Delta$ induces an automorphism of its skeletal building, hence we have a natural monomorphism
	\[\iota \colon \Aut(\Delta) \hookrightarrow \Aut(\Phi).\]
	Let $g\in\U_\Delta(\F) \leq \Aut(\Delta)$, let $\R$ be any panel of $\Phi$ of type $\ell$, and consider the local action of $\iota(g)$ as an automorphism of $\Phi$ at the panel $\R$. For readability, we will identify $g$ with its image $\iota(g)$. We can also identify $\R$ with a residue of $\Delta$ of type $I_\ell$ (which is isomorphic to $\Delta_\ell$). Then the local action
	\[\sigma_\varphi(g,\R)
		= \restrict{\varphi_\ell}{g\acts\R} \circ \restrict{g}{\R} \circ \restrict[-1]{\varphi_\ell}{\R}\]
	is the composition of three isomorphisms $\Delta_\ell \to \R \to g\acts\R \to \Delta_\ell$ and is hence an automorphism of the building $\Delta_\ell$. It thus makes sense to consider the local action of $\sigma_\varphi(g,\R)$ at an $i$-panel $\P$ of $\Delta_\ell$ with $i\in I_\ell$. Let $\P' = \restrict[-1]{\varphi_\ell}{\R}(\P)$, which is an $i$-panel of $\Delta$ by \cref{lem:cityproductresidue}.
	We then have $\P' \subseteq \R \subseteq \Delta$, and the map $\varphi_\ell \colon \Delta \to \Delta_\ell$ restricts to isomorphisms $\restrict{\varphi_\ell}{\R} \colon \R \to \Delta_\ell$ and $\restrict{\varphi_\ell}{\P'} \colon \P' \to \P$.
	Then
	\begin{align*}
	\label{eq:locallocal}
		\sigma_{\lambda^\ell}(\sigma_\varphi(g,\R),\P)
		& = \restrict{\lambda^\ell_i}{\sigma_\varphi(g,\R)\acts\P} \ \circ\ 
			\restrict{\sigma_\varphi(g,\R)}{\P} \ \circ\ 
			\restrict[-1]{\lambda^\ell_i}{\P}\\
		& = \restrict{\lambda^\ell_i}{\sigma_\varphi(g,\R)\acts\P} \ \circ\ 
			\restrict{\bigl(\restrict{\varphi_\ell}{g\acts\R} \circ \restrict{g}{\R} \circ \restrict[-1]{\varphi_\ell}{\R}\bigr)}{\P} \ \circ\ 
			\restrict[-1]{\lambda^\ell_i}{\P}\\
		& = \restrict{\lambda^\ell_i}{\sigma_\varphi(g,\R)\acts\P} \ \circ\ 
			\restrict{\varphi_\ell}{g\acts\P'} \circ \restrict{g}{\P'} \circ \restrict[-1]{\varphi_\ell}{\P'} \ \circ\ 
			\restrict[-1]{\lambda^\ell_i}{\P}\\
		& = \restrict{\lambda'_i}{g\acts\P'}
			\circ \restrict{g}{\P'}
			\circ \restrict[-1]{\lambda'_i}{\P'}\\
		& = \sigma_{\lambda'}(g,\P').
	\tag{$\ast$}
	\end{align*}
    %
	\[\begin{tikzcd}[dims={9em}{4.2em}]
		\Delta \ar[r,"g"{name=A}] \ar[d,equal] \ar[ddd,"{\lambda'_i}",swap,relay arrow=-10mm]
			& \Delta \ar[d,equal] \ar[ddd,"{\lambda'_i}",relay arrow=10mm]\\
		\Phi \ar[r,"\iota(g)"{name=B},swap] \ar[d,"\varphi_\ell",swap]
			& \Phi \ar[d,"\varphi_\ell"]\\
		\Delta_\ell \ar[r,"{\sigma_\varphi(g,\,{\bullet})}"] \ar[d,"\lambda^\ell_i",swap]
			& \Delta_\ell \ar[d,"\lambda^\ell_i"]\\
		\Omega_i \ar[r,"{\sigma_{\lambda^\ell}(\sigma_\varphi(g,\,{\bullet}),\,{\bullet})}"]
			& \Omega_i
		\arrow[from=A, to=B, Rightarrow, shorten <=12pt, shorten >=10pt, "\strut\iota"]
	\end{tikzcd}\]
	\smallskip
	
	Since $g\in\U_\Delta(\F)$, the result of \cref{eq:locallocal} is a permutation in $F_i$, so we conclude that $\sigma_\varphi(g,\R) \in \U_{\Delta_\ell}(\F_\ell)$. Since this holds for every $\ell$-panel $\R$ of $\Phi$, this shows that $\iota(g) \in \U_\Phi(\boldsymbol{F'})$.
	
	\medskip

	Conversely, let $g\in\U_\Phi(\boldsymbol{F'})$. We can identify $g \in \Aut(\Phi)$ with a permutation of~$\Delta$, and we first verify that this permutation is type-preserving, i.e., that $g$ is in fact an automorphism of $\Delta$. Indeed, let $c\sim_i d$ be $i$-adjacent chambers in $\Delta$. Let $\ell := \ell(i)$ and let $\R$ be the residue of $\Delta$ of type $I_\ell$ containing $c$ and $d$. Then $\R$ is an $\ell$-panel of $\Phi$. The local action $\sigma_\varphi(g,\R)$ is an element of $F'_\ell = \U_{\Delta_\ell}(\F_\ell) \leq \Aut(\Delta_\ell)$. Hence
	\[\restrict{g}{\R} = \restrict[-1]{\varphi_\ell}{g\acts\R} \circ \sigma_\varphi(g,\R) \circ \restrict{\varphi_\ell}{\R}\]
	is a composition of isomorphisms $\R \to \Delta_\ell \to \Delta_\ell \to g\acts\R$, each of which preserves $i$-adjacency. In particular, $g\acts c\sim_i g\acts d$. Since $c$ and $d$ were arbitrary, we conclude that $g\in\Aut(\Delta)$.
	
	Next, let $\P'$ be any $i$-panel in $\Delta$ with $i\in I_\ell$, let $\R$ be the $I_\ell$-residue of $\Delta$ containing $\P'$ and let $\P := \varphi_\ell(\P')$ in $\Delta_\ell$. The reverse calculation of \cref{eq:locallocal} shows that the local action satisfies
	\[\sigma_{\lambda'}(g,\P') = \sigma_{\lambda^\ell}(\sigma_\varphi(g,\R),\P).\]
	Since $\sigma_\varphi(g,\R) \in \U_{\Delta_\ell}(\F_\ell)$, we have $\sigma_{\lambda'}(g,\P') \in (\F_\ell)_i = F_i$. We conclude that indeed $g\in \U_\Delta(\F)$.
	
	\medskip
	
	In conclusion, the restriction of $\iota \colon \Aut(\Delta) \hookrightarrow \Aut(\Phi)$ to $\U_\Delta(\F) \leq \Aut(\Delta)$ is an isomorphism $\kappa \colon \U_\Delta(\F) \to \U_\Phi({\boldsymbol F'})$.
	Finally, notice that it is obvious that $\iota$ is a \emph{homeomorphism} onto its image because $\Delta$ and $\Phi$ have the same underlying set, and the topology on $\Aut(\Delta)$ and $\Aut(\Phi)$ is independent of the additional building structure on $\Delta$ and $\Phi$.
	In particular, $\kappa$ is an isomorphism of topological groups.
\end{proof}

\section{Application: Different right-angled buildings of the same type with isomorphic universal groups}

Inspired by \cite{bessmann}, we provide a construction to produce pairs of right-angled buildings of the same type $M$ over $I$ but with different parameters $(q_i)_{i \in I}$, that nevertheless admit isomorphic universal groups for appropriate choices of the local data $\F$.    

The following result is certainly not the most general result possible, but it seems a good trade-off between producing a large amount of examples and still being ``readable''.

\begin{theorem}\label{thm:application}
    Let $M$ be a right-angled diagram of rank $n$ admitting a non-trivial symmetry $\rho \in \Sym(n)$, and assume that $k \in \{ 1,\dots,n \}$ is not fixed by $\rho$. Let $\Lambda$ be the support of $\rho$ (i.e., the set of elements not fixed by $\rho$).
    For each $\ell \in \{ 1,\dots, n \}$, let $M_\ell$ be a right-angled diagram over $I_\ell$, such that
    \begin{itemize}
        \item $|I_\ell| = 1$ for each $\ell \in \Lambda \setminus \{ k \}$, and
        \item $|I_k| =: t > 1$. Write $I_k = \{ i_1,\dots,i_t \}$.
    \end{itemize}
    Consider the city product $N := \cityprod_M(M_1,\dots,M_n)$, with index set $I = \bigsqcup_{\ell=1}^{n} I_\ell$.
    \begin{itemize}
        \item For each $i \in I_k$, let $G_i$ and $G'_i$ be two arbitrary permutation groups (acting on sets $\Omega_i$ and $\Omega'_i$, respectively).
        \item For each $i \in I_\ell$ for $\ell \not\in \{ k, \rho(k) \}$, let $H_i$ be an arbitrary permutation group (acting on a set $\Omega_i$).
    \end{itemize}
    Finally, consider the two collections $\F$ and $\F'$ of local data over $I$ defined by
    \begin{itemize}
        \item $F_i := G_i$ for $i \in I_k$,
        \item $F_i := \U_{M_k}(G'_{i_1},\dots,G'_{i_t})$ for the unique $i \in I_{\rho(k)}$,
        \item $F_i := H_i$ for all other $i \in I$, \\[-1.8ex]
        \item $F'_i := G'_i$ for $i \in I_k$,
        \item $F'_i := \U_{M_k}(G_{i_1},\dots,G_{i_t})$ for the unique $i \in I_{\rho^{-1}(k)}$,
        \item $F'_i := H_j$ for all $i \in I_\ell$ with $\ell \in \Lambda \setminus \{ k, \rho^{-1}(k) \}$, where $I_{\rho(\ell)} = \{ j \}$,
        \item $F'_i := H_i$ for all other $i \in I$.
    \end{itemize}
    Then the universal groups $\U_N(\F)$ and $\U_N(\F')$ are isomorphic (as topological groups).
\end{theorem}
\begin{proof}
    We will use the notation from \itemref{rem:Unotation}{2}.
    Consider the collections $\bL$ and $\bL'$ of local data over $\{ 1,\dots,n \}$ defined by
    \begin{align*}
        L_k &= \U_{M_k}(G_{i_1},\dots,G_{i_t}), \\
        L_{\rho(k)} &= \U_{M_k}(G'_{i_1},\dots,G'_{i_t}), \\
        L_\ell &= H_i \quad \text{for all } \ell \in \Lambda \setminus \{ k, \rho(k) \}, \text{ where } I_\ell = \{ i \} , \\
        L_\ell &= \U_{M_\ell}(H_i \mid i \in I_\ell) \quad \text{for all } \ell \notin \Lambda, \\[1.6ex]
        L'_k &= \U_{M_k}(G'_{i_1},\dots,G'_{i_t}), \\
        L'_{\rho^{-1}(k)} &= \U_{M_k}(G_{i_1},\dots,G_{i_t}), \\
        L'_\ell &= H_j \quad \text{for all } \ell \in \Lambda \setminus \{ k, \rho^{-1}(k) \}, \text{ where } I_{\rho(\ell)} = \{ j \} , \\
        L'_\ell &= \U_{M_\ell}(H_i \mid i \in I_\ell) \quad \text{for all } \ell \notin \Lambda.
    \end{align*}
    By \cref{thm:universalcityproduct}, we then have
    \begin{align*}
        \U_N(\F) &\cong \U_M(\bL) \quad \text{and} \\
        \U_N(\F') &\cong \U_M(\bL').
    \end{align*}
    Since $L'_\ell = L_{\rho(\ell)}$ for all $\ell$, it is now immediately clear that the symmetry $\rho$ of the diagram $M$ ensures that $\U_M(\bL) \cong \U_M(\bL')$, and the result follows.
\end{proof}

\begin{examples}\label{ex:isom}
\begin{enumerate}[label={\rm (\arabic*)}]
    \item\label{ex:isom:1}
        The first set of examples covers precisely the cases that can also be obtained with Lara Be\ss man's method from \cite{bessmann}.
        
        Let $M$ be the diagram of rank $2$ with label $\infty$, let $M_1$ be the diagram of rank $1$ and let $M_2$ be the diagram of rank $t > 1$ without edges.
        The city product $N := \cityprod_M(M_1, M_2)$ of these diagrams has rank $t+1$; we label the $t+1$ nodes of $N$ as below.
        
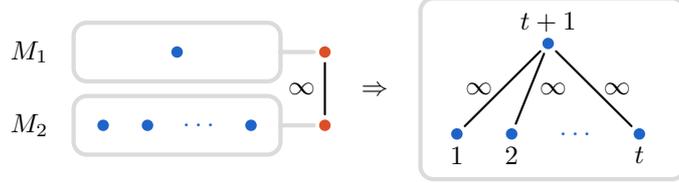
\[
        \begin{tikzpicture}[scale=.81]
			\path (0,0) node[myvertex,ugentblue] (A1) {}
				+(.6,0) node[myvertex,ugentblue] (A2) {}
				+(1.32,0) node[ugentblue] (A3) {$\dots$}
				+(2,0) node[myvertex,ugentblue] (A4) {}
				++(1,1) node[myvertex,ugentblue] (A5) {}
				(3,0) node[myvertex,ugentred] (B1) {}
				++(0,1) node[myvertex,ugentred] (B2) {};
			\draw[myedge,ultra thick,mydarkgray,rounded corners=5pt]
				(-.4,-.4) rectangle (2.4,.4) (2.4,0) -- (B1)
				(-.4,.6) rectangle (2.4,1.4) (2.4,1) -- (B2);
			\draw[myedge] (B1) -- node[left] {$\infty$} (B2);
			\node[] at (-1,1) {$M_1$};
			\node[] at (-1,0) {$M_2$};
		\end{tikzpicture}
		\hspace*{2ex} \Rightarrow \hspace*{2ex}
		\begin{tikzpicture}
			\path (0,0) node[myvertex,ugentblue] (A1) {}
				+(.6,0) node[myvertex,ugentblue] (A2) {}
				+(1.32,0) node[ugentblue] (A3) {$\dots$}
				+(2,0) node[myvertex,ugentblue] (A4) {}
				++(1,1) node[myvertex,ugentblue] (B1) {};
			\draw[myedge,ultra thick,mydarkgray,rounded corners=5pt] (-.4,-.5) rectangle (2.5,1.5);
			\draw[myedge] (A1) -- node[left] {$\infty$} (B1) -- node[right] {$\infty$} (A2);
			\draw[myedge] (A4) -- node[right] {$\infty$} (B1);
			\node[] at ([yshift=.8em] B1) {$t+1$};
			\node[] at ([yshift=-.8em] A1) {$1$};
			\node[] at ([yshift=-.8em] A2) {$2$};
			\node[] at ([yshift=-.8em] A4) {$t$};
		\end{tikzpicture}
        \]
        We can now apply \cref{thm:application}, with $\rho$ the unique non-trivial symmetry of $M$ and with $k=2$, so $\rho(k) = \rho^{-1}(k) = 1$ and $|I_k| = t > 1$.
        For each $i \in \{ 1,\dots, t \}$, let $G_i$ and $G'_i$ be two arbitrary permutation groups (acting on sets $\Omega_i$ and $\Omega'_i$, respectively).
        Notice that there are no values $\ell \notin \{ k, \rho(k) \}$, so we do not have to choose groups $H_i$ as in the theorem.
        Moreover, notice that because $M_2$ has no edges, we simply have $\U_{M_2}(G_1,\dots,G_t) \cong G_1 \times \dots \times G_t$ (see \cref{lem:universalreducible}).
        
        It now follows from \cref{thm:application} that
        \[ \U_N(G_1,\ G_2,\ \dots,\ G_t,\ G'_1 \times \dots \times G'_t) \cong \U_N(G'_1,\ G'_2,\ \dots,\ G'_t,\ G_1 \times \dots \times G_t) . \]
    \item\label{ex:isom:2}
        We now present an example with a non-involutory symmetry.
        Let $M$ be the diagram of rank $3$ with all labels $\infty$, let $M_1$ and $M_3$ be diagrams of rank $1$ and let $M_2$ be the diagram of rank $2$ without edges.
        The city product $N := \cityprod_M(M_1, M_2, M_3)$ of these diagrams has rank $4$, labeled as below.
        
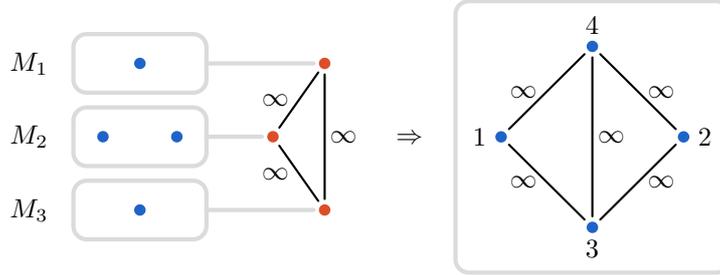
\[
        \begin{tikzpicture}[scale=.81]
			\path (.5,-1) node[myvertex,ugentblue] (A1) {}
			    (0,0) node[myvertex,ugentblue] (A2) {}
				+(1,0) node[myvertex,ugentblue] (A3) {}
				++(.5,1) node[myvertex,ugentblue] (A4) {}
				(3,-1) node[myvertex,ugentred] (B1) {}
				+(-.7,1) node[myvertex,ugentred] (B2) {}
				+(0,2) node[myvertex,ugentred] (B3) {};
			\draw[myedge,ultra thick,mydarkgray,rounded corners=5pt]
				(-.4,-.4) rectangle (1.4,.4) (1.4,0) -- (B2)
				(-.4,.6) rectangle (1.4,1.4) (1.4,1) -- (B3)
				(-.4,-1.4) rectangle (1.4,-.6) (1.4,-1) -- (B1);
			\draw[myedge] (B1) -- node[left] {$\infty$} (B2) -- node[left] {$\infty$} (B3) -- node[right] {\!$\infty$} (B1);
			\node[] at (-1,1) {$M_1$};
			\node[] at (-1,0) {$M_2$};			
			\node[] at (-1,-1) {$M_3$};			
		\end{tikzpicture}
		\hspace*{2ex} \Rightarrow \hspace*{2ex}
		\begin{tikzpicture}
			\path (1,0) node[myvertex,ugentblue] (A1) {}
				+(-1,1) node[myvertex,ugentblue] (A2) {}
				+(1,1) node[myvertex,ugentblue] (A3) {}
				+(0,2) node[myvertex,ugentblue] (A4) {};
			\draw[myedge,ultra thick,mydarkgray,rounded corners=5pt] (-.5,-.5) rectangle (2.5,2.5);
			\draw[myedge] (A1) -- node[left] {$\infty$} (A2) -- node[left] {$\infty$} (A4);
			\draw[myedge] (A1) -- node[right] {$\infty$} (A3) -- node[right] {$\infty$} (A4);
			\draw[myedge] (A1) -- node[right] {\!$\infty$} (A4);
			\node[] at ([yshift=-.8em] A1) {$3$};
			\node[] at ([xshift=-.8em] A2) {$1$};
			\node[] at ([xshift=.8em] A3) {$2$};
			\node[] at ([yshift=.8em] A4) {$4$};
		\end{tikzpicture}
        \]
        We now choose $\rho$ to be the cyclic symmetry $(123)$ of $M$ of order $3$ and we choose $k=2$, so $\rho(k)=3$ and $\rho^{-1}(k)=1$.
        We let $G_1,G_2,G'_1,G'_2,H$ be five permutation groups (acting on sets $\Omega_1,\Omega_2,\Omega'_1,\Omega'_2,\Omega$, respectively).
        
        It now follows from \cref{thm:application} that
        \[ \U_N(G_1,\ G_2,\ G'_1 \times G'_2,\ H) \cong \U_N(G'_1,\ G'_2,\ H,\ G_1 \times G_2) . \]        
\end{enumerate}    
\end{examples}

Notice that both examples can exist for \emph{locally finite} buildings (i.e., buildings where each of the parameters $q_i$, $i \in I$ is finite). This happens because the diagram $M_k$ has no edges, so that the corresponding universal group $\U_{M_k}(G_1,\dots,G_t)$ is just a direct product of the groups $G_i$. On the other hand, if the diagram $M_k$ has edges, then the universal group $\U_{M_k}(G_1,\dots,G_t)$ is never finite, so those examples do not occur in the locally finite case.


\clearpage

\nocite{*}
\footnotesize
\bibliographystyle{alpha}
\bibliography{sources}

\bigskip

\end{document}